\documentclass[12pt]{amsart}
\usepackage{latexsym}
\usepackage{amssymb}
\usepackage{verbatim}
\usepackage{tikz}
\usepackage{enumerate,framed}

\usepackage{amsmath}

\newcommand{\ignore}[1]{}

\newcommand{\hide}[1]{}

\DeclareMathOperator{\Der}{Der}

\newcommand{\N}{\mathbb N}

\newcommand{\Z}[0]{\mathbb Z}

\newtheorem{dummy}{Dummy}

\newtheorem{lemma}[dummy]{Lemma}

\newtheorem{thm}[dummy]{Theorem}

\newtheorem{prop}[dummy]{Proposition}

\theoremstyle{definition}
\newtheorem{definition}[dummy]{Definition}

\theoremstyle{remark}

\newtheorem*{rem*}{Remark to ourselves}

\begin{document}

\bibliographystyle{alpha}

\title{Graded Lie algebras of maximal class of type $n$}

\author{Sandro Mattarei}
\address{Charlotte Scott Research Centre for Algebra\\
University of Lincoln \\
Brayford Pool
Lincoln, LN6 7TS\\
United Kingdom}
\email{smattarei@lincoln.ac.uk}

\author{Simone Ugolini}
\address{Dipartimento di Matematica\\
Universit\`a degli Studi di Trento\\
via Sommarive 14\\
I-38123 Povo (Trento)\\
Italy}
\email{s.ugolini@unitn.it}



\begin{abstract}
Let $n>1$ be an integer.
The algebras of the title, which we abbreviate as {\em algebras of type $n$,}
are infinite-dimensional graded Lie algebras
$
L= \bigoplus_{i=1}^{\infty}L_i,
$
which are generated by an element of degree $1$ and an element of degree $n$, and satisfy
$[L_i,L_1]=L_{i+1}$ for $i\ge n$.
Algebras of  type $2$ were classified
by Caranti and Vaughan-Lee in 2000 over any field of odd characteristic.
In this paper we lay the foundations for a classification
of algebras of arbitrary type $n$,
over fields of sufficiently large characteristic relative to $n$.
Our main result describes precisely all possibilities
for the first constituent length
of an algebra of type $n$, which is a numerical invariant closely related to the dimension of its
largest metabelian quotient.
\end{abstract}

Author Accepted Manuscript

Published in Journal of Algebra \textbf{593} (2022), 142--177.

\verb#https://dx.doi.org/10.1016/j.jalgebra.2021.11.012#

\bigskip

\maketitle

\section{Introduction}\label{sec:intro}

In this paper an {\em algebra (of maximal class) of type $n$,}
where $n>1$ is an integer, is a graded Lie algebra
$L=\bigoplus_{i=1}^{\infty}L_i$ over a field,
with $\dim L_i=0$ for $1<i<n$ and $\dim L_i=1$ otherwise,
such that $[L_i,L_1]=L_{i+1}$ for all $i\ge n$.
In particular, according to this definition $L$ has infinite dimension.
The qualifier `of maximal class', which we will omit as a rule,
refers to the fact that each Lie power $L^i$ of $L$ with $i>2$
has codimension one in the previous Lie power $L^{i-1}$,
and hence the quotient $L/L^i$ is nilpotent of nilpotency class $i-1$,
thus one less than its dimension.
Motivations for investigating such algebras can be ultimately
traced back to the {\em coclass theory} of pro-$p$ groups,
but we refer to the Introduction of~\cite{IMS} for a discussion
of those origins.

Over a field of characteristic zero a classification of algebras of type $2$
was obtained by Shalev and Zelmanov in~\cite{ShZe:narrow-Witt},
within a broader investigation.
In fact,~\cite[Theorem~7.1]{ShZe:narrow-Witt} showed, among else, that
over a field of characteristic zero there are precisely
three algebras of type $2$, up to isomorphism of graded Lie algebras.
One of them is the positive part of the Witt algebra
$\Der\bigl(F[X]\bigr)$,
hence the subalgebra $W$ with graded basis given by
$E_i=x^{i+1}\partial/\partial x$
for $i>0$ (having degree $i$), with Lie product
$[E_i,E_j]=(j-i)E_{i+j}$.
The other two are soluble, and have a maximal abelian ideal of codimension one and two, respectively
(see Section~\ref{sec:sequence}).

Each of those three Lie algebras of type $2$ makes sense
in positive characteristic $p$, but only the two soluble ones
are algebras of type $2$
(due to $[E_{p+1},E_1]=0$ in the Witt algebra).
Caranti and Vaughan-Lee classified algebras of type $2$
over a field of odd characteristic in~\cite{CVL00},
and over a field of characteristic two in~\cite{CVL03},
but their classification involves another degree of complexity due
to a connection with algebras of type $1$, which we now introduce.

For the purposes of this paper, again limiting ourselves to
infinite-dimensional algebras,
an {\em algebra (of maximal class) of type $1$}
is a graded Lie algebra
$L=\bigoplus_{i=1}^{\infty}L_i$ over a field,
with $\dim L_1=2$,
$\dim L_i=1$ for $i>1$,
and such that $[L_i,L_1]=L_{i+1}$ for all $i\ge 1$.
In particular, such algebras are generated by their
first homogeneous component $L_1$, thus justifying their name,
whilst algebras of type $n$ are clearly generated by an element of degree $1$ and one of degree $n$.

It is not hard to see that over a field of characteristic zero
all algebras of type $1$ are isomorphic, and have an abelian maximal ideal.
In~\cite{Sha:max} Shalev constructed a countable family of
insoluble algebras of type $1$, over an arbitrary field of prime characteristic.
A systematic study of algebras of type $1$ started with~\cite{CMN}.
There it was shown that over any field $F$ of positive characteristic there are $|F|^{\aleph_0}$
non-isomorphic algebras of type $1$.
This large cardinality arose from possibly repeating
certain procedures countably many times.
In~\cite{CN} Caranti and Newman classified all algebras of type $1$
over a field of odd characteristic, by showing that
suitable application of the machinery developed in~\cite{CMN} is able to account for all of them.
A similar result in characteristic $2$ followed in~\cite{Ju:maximal}.
Thus, algebras of type $1$ have been classified, although the classification statement
may look unconventional as it involves the possible application of certain steps countably many times.
This classification is also important because classifications of other classes of Lie algebras rely on it as an essential ingredient.
In particular, this is the case for the algebras
of type $n$
studied in this paper, and also for {\em thin} Lie algebras,
see~\cite{AviMat:A-Z} and the references therein.

A classification of algebras of type $n$ is bound to
incorporate all the complexity of algebras of type $1$ due to the following simple observation.
An algebra of type $1$ is called {\em uncovered} if there is an element $z\in L_1$ such that $[L_i,z]=L_{i+1}$ for all $i>1$.
This condition can always be achieved by extending the ground field, and is automatically satisfied if that is uncountable.
If $L$ is an uncovered algebra of type $1$, then the subalgebra generated by $z$ and $L_n$ is an algebra of type $n$.
Different choices of $z$ in $L_1$ may lead to non-isomorphic subalgebras of type $n$.
Thus, in any prospective classification of algebras of type $n$
we should expect subalgebras of algebras of type $1$
to form a substantial ingredient.

Caranti and Vaughan-Lee obtained a full classification of algebras of type $2$ in~\cite{CVL00,CVL03}.
In particular, in~\cite{CVL00} they showed that
any algebra of type $2$ over a field of odd characteristic,
which is not a graded subalgebra of an algebra of type $1$ in the way described,
belongs to an explicit countable family of soluble algebras,
with an additional family appearing only in characteristic $3$.
A classification of algebras of type $2$ in characteristic $2$ was then
obtained in~\cite{CVL03}.
The differences between~\cite{CVL00} and~\cite{CVL03} suggested a
distinction, in methods and conclusions, between
the tasks of classifying algebras of type $n$ with $n$
less than the characteristic $p$, and algebras of type $p$.
Progress on each of those more general problems
was made in the PhD theses of Ugolini~\cite{Ugo:thesis}
and Scarbolo~\cite{Sca:thesis} (under the supervision
of the first author, jointly with Caranti in the former case).

A crucial concept for the study of algebras of type $1$
in~\cite{CMN,CN}, that of constituents and their lengths,
was extended to algebras of type $2$ in~\cite{CVL00},
and slightly updated in~\cite{CVL03} to better include characteristic $2$.
Following the latter, constituents for algebras of arbitrary type $n>1$
were defined in~\cite[Section~3]{IMS}.
We postpone recalling the detailed definition to
Section~\ref{sec:constituents},
and content ourselves in this introduction to quoting
from~\cite[Section~5]{IMS} a convenient characterization of their lengths
in terms of the Lie powers (or terms of the lower central series)
of the Lie square $L^2$ of $L$.

Assuming $[L_{n+1},L_n]=0$,
the {\em constituent lengths} of an algebra $L$ of type $n>1$
are given by
$\ell=\ell_1=\dim\bigl(L^2/(L^2)^2\bigr)+n$,
and
$\ell_r=\dim\bigl((L^2)^r/(L^2)^{r+1}\bigr)$ for $r>1$.
These equations are valid also for algebras of type $1$,
and unconditionally in that case.
However, for $n>1$ they do require the hypothesis $[L_{n+1},L_n]=0$,
which is easily seen to be equivalent to $[L^2,L_n]\subseteq L^{n+3}$.
The natural definition of constituents given in~\cite[Section~3]{IMS}
implies $\ell\ge 2n$, and $\ell$ takes its minimal allowed value $2n$
precisely when $[L_{n+1},L_n]\neq 0$.
In that case, however,
one loses the direct connection stated above
with the dimension of $L/(L^2)^2$, the largest metabelian quotient of $L$.
For example, it is easy to see (\cite[Section~2]{IMS}) that,
over a field of arbitrary characteristic,
for $n>1$ there are precisely
two metabelian algebras of type $n$ up to isomorphism.
However, one has $\ell=\infty$ and the other has $\ell=2n$,
although $(L^2)^2=0$ in both cases.
In fact, in the exceptional situation where  $[L_{n+1},L_n]\neq 0$
the constituents of an algebra $L$ of type $n$
are actually of rather limited use,
as was already noted in the case $n=2$:
according to~\cite{CVL00}
such algebras ``do not admit a theory of constituents,''
and they had to be dealt with in an {\em ad hoc} manner there.
In the main result of this paper, which we describe now,
we will actually need a stronger assumption
on $\ell$ than just $\ell>2n$ to avoid problematic settings.

The first constituent length $\ell$ is the main structural parameter
of an algebra of type $n$.
Finding the possible values it may take is a fundamental
starting point of a classification work.
In particular, according to~\cite[Theorem~5.5]{CMN}
(or the revisited proof in~\cite{Mat:chain_lengths}),
the first constituent length $\ell$ of a non-metabelian
algebra of type $1$ always equals $2q$ for some power
$q$ of the characteristic $p$.
When $n>1$ small values of $\ell$ cause complications,
and hence we will assume $\ell>4p$ here
for cleaner proofs and statements.
The main goal of this paper is proving an analogous result
for algebras of type $n$ with $1<n<p$,
which is part of the following result.

\begin{thm}\label{thm:first_length}
Let an algebra of type $n$, over a field of odd characteristic $p$, with $1<n\le p$,
have its first constituent of finite length $\ell>4p$.
Then
\begin{itemize}
\item either $\ell=2q$, where $q$ is a power of $p$,
\item or $q<\ell\le q+n$, where $q$ is a power of $p$ and $\ell$ is even.
\end{itemize}
\end{thm}

In the special case $n=2$ Theorem~\ref{thm:first_length}
says  $\ell$
equals either $2q$ or $q+1$ for some power $q$ of $p$ if $p>2$
(as in~\cite[Section~4]{CVL03}),
and equals either $2q$ or $q+2$ if $p=2$
(as in~\cite[Proposition~3.1]{CVL03},
with the latter case equivalently written in the form $2q'+2$ there).

The special case of Theorem~\ref{thm:first_length} where $n=p$
was obtained in~\cite{Sca:thesis}, and appears
as~\cite[Theorem~2]{IMS} with a new proof.
In fact, extending the results of~\cite{CVL03},
algebras of type $p$ were fully classified
in~\cite{Sca:thesis} under the assumption $\ell>4p$
(erroneously omitted in~\cite{Sca:thesis},
but required for that proof).
An exposition of some ingredients of the classification proof
for $n=p$ is given in~\cite{IMS}.

When Theorem~\ref{thm:first_length} is combined with the results
of~\cite{Ugo:thesis}, it yields the following classification of algebras of type $n$, under an additional assumption $p>4n$.

\begin{thm}\label{thm:classification}
Let $L$ be an algebra of type $n > 1$ over a field of characteristic $p$.
If $p>4n$ and $\ell>4p$, then $L$ is either a graded subalgebra of an uncovered algebra of type $1$
or $L$ belongs to an explicitly described countable family $\mathcal{E}$.
\end{thm}

In fact, Ugolini~\cite{Ugo:thesis}
classified algebras of type $n$ with first constituent length
$q<\ell \le q+n$ or $\ell = 2q$, for any power $q$ of $p$
but under the hypothesis $p>4n$.
However, the above classification result of algebras of type $n$
could not be established in~\cite{Ugo:thesis}.
That is because
only a weaker version of Theorem~\ref{thm:first_length} was attained in~\cite{Ugo:thesis},
allowing a further range of possibilities for $\ell$
which cannot really occur for algebras of type $n$.
In fact, excluding that spurious range of values for $\ell$
is the hardest and most technical part of our proof of
Theorem~\ref{thm:first_length}.
That was already the case in the special case $n=p$
of Theorem~\ref{thm:first_length} proved in~\cite{IMS}, but much more so
in the case $n<p$ proved here, where it occupies
a large part of Section~\ref{sec:first_length_proof},
and the whole long
Section~\ref{sec:final_contradiction}.
It appears likely that the results of~\cite{Ugo:thesis}
should remain true after replacing the hypothesis $4n<p$ with
$n<p$ and $\ell>4p$.
If that is confirmed then Theorem~\ref{thm:first_length}
would hold under the weaker hypotheses
$n<p$ and $\ell>4p$.

Each of the possibilities for $\ell$ listed in Theorem~\ref{thm:first_length} does occur for some of the algebras of Theorem~\ref{thm:classification}.
In particular, the first constituent of $L$ has length $\ell=2q$
whenever $L$ is a graded subalgebra of an uncovered algebra of type $1$.
The remaining possibilities occur for the algebras of the exceptional family $\mathcal{E}$.
In this paper we also give an explicit construction
for the algebras of the family $\mathcal{E}$.
We summarize here our conclusions as an existence result, and refer to Section~\ref{sec:exceptional} for more precise
structural information on those algebras.

\begin{thm}\label{thm:exceptional}
Let $q$ be a power of an odd prime $p$, with $q>p$, and let $1<n<p$.
For every $0<m<n$ there is a soluble algebra $L$ of type $n$, over a field of odd characteristic $p$,
satisfying the following:
\begin{itemize}
\item
the length $\ell$ of the first constituent of $L$ equals $q+m$ if $m$ is odd, and $q+m+1$ if $m$ is even;
\item
all constituents past the first have length $q$,
except for the second constituent if $m$ is even, which has then length $q-1$.
\end{itemize}
\end{thm}

Note that the algebras of Theorem~\ref{thm:exceptional} attain all even values of $\ell$ in the interval
$q<\ell\le q+n$ of~Theorem~\ref{thm:first_length}.

We now outline the structure of the paper.
In Sections~\ref{sec:sequence} and~\ref{sec:constituents}
we recall definitions pertaining to algebras of type $n$,
most notably constituents, and recall some of their basic structural properties,
referring to~\cite[Sections~2--5]{IMS} for a broader discussion and proofs.

Our original contributions start with Section~\ref{sec:polynomials},
which establishes a technical result in a polynomial context.
In fact, the central argument of the proof of
Theorem~\ref{thm:first_length}, in Subsection~\ref{subsec:main},
translates a sequence of Lie algebraic
equations which hold in an algebra of type $n$
into a polynomial condition, which is more easily and tidily
dealt with separately.
This general scheme of proof was employed in~\cite{IMS}
for algebras of type $p$, with the polynomial approach
built on an idea first introduced in~\cite{Mat:chain_lengths}.
Section~\ref{sec:polynomials} resembles~\cite[Section~6]{IMS},
but various adaptations were needed to go from the case $n=p$
of~\cite{IMS} to the case $1<n<p$ of interest here.

The polynomial results of Section~\ref{sec:polynomials},
when applied in Subsection~\ref{subsec:main} to an algebra $L$ of type $n$
with first constituent length $\ell$,
restrict the possibilities for $\ell$ to those listed in
Theorem~\ref{thm:first_length}, plus a further range of spurious values.
Ruling out the latter requires further Lie algebra calculations
and forms the longest part of the proof of Theorem~\ref{thm:first_length},
in fact much longer than in the broadly similar proof in~\cite{IMS}.
The main part of that argument is addressed in
Subsection~\ref{subsec:spurious}, leaving us with a special configuration
which is particularly hard to rule out.
Because of the even more technical nature of the proof we postpone
dealing with that configuration to Section~\ref{sec:final_contradiction}.

Finally, in Section~\ref{sec:exceptional} we present a construction
of the algebras of the exceptional family $\mathcal{E}$ of
Theorem~\ref{thm:classification}.
Our construction works under more general hypotheses than those
of Theorem~\ref{thm:classification}.
It extends a construction given in~\cite[Section~10]{IMS}
for the case $n=p$, based on derivations of a ring of divided powers,
and is more efficient and compact
than the original construction of those algebras
given in~\cite{Ugo:thesis}.
In particular, our construction yields Theorem~\ref{thm:exceptional},
but also much more precise information on those algebras.
In particular, full structural information on those algebras
(that is, their multiplication table if we wish)
is condensed in certain formal power series,
for which we exhibit explicit expressions as rational functions
in Proposition~\ref{prop:gen_func}.

\section{Algebras of type $n$}\label{sec:sequence}

In this section we introduce the Lie algebras
of interest for the present study,
referring to~\cite[Section~2]{IMS} for a broader review
and more detailed explanations.

A systematic study of graded Lie algebras of maximal class
begun in~\cite{CMN}.
Those were defined as graded Lie algebras
$L=\bigoplus_{i=1}^{\infty}L_i$, over a field,
with $\dim L_1=2$, $\dim L_i\le 1$ for $i>1$,
and $[L_i,L_1]=L_{i+1}$ for all $i$.
Graded Lie algebras of maximal class according
to this definition are generated by $L_1$.
This was a natural assumption coming to graded Lie algebras
from the study of $p$-groups (and pro-$p$ groups),
being satisfied by the graded Lie algebra
associated to the lower central series of a $p$-group.
Because Lie algebras of maximal class with other gradings
are of interest, we will refer to those graded Lie algebras
of maximal class defined above as {\em algebras of type $1$}.
Also, we will only be interested in infinite-dimensional
Lie algebras in this paper, and we take infinite-dimensionality
of all Lie algebras we mention from now on as a blanket assumption.
Thus, an algebra of type $1$ is a graded Lie algebra
$L=\bigoplus_{i=1}^{\infty}L_i$, over a field,
with $\dim L_1=2$, $\dim L_i=1$ for $i>1$,
and $[L_i,L_1]=L_{i+1}$ for all $i$.
As we mentioned in the Introduction, over any field
of positive characteristic there are very many
algebras of type $1$, but they may be considered classified
according to~\cite{CN} in odd characteristic, and~\cite{Ju:maximal}
in characteristic $2$.

A basic tool for the study of algebras of type $1$,
in~\cite{CMN,CN,Ju:maximal}, was the sequence of two-step centralizers.
This terminology goes back to early work on $p$-groups of maximal class
started by Norman Blackburn in 1958, but because this connection
disappears for algebras of type $n>1$, we will simply refer to it
as the {\em sequence} of an algebra of type $1$.
In the original definition (see~\cite[Section~2]{IMS}) the terms of
this sequence are elements of a projective line over the base field $F$,
but we will place ourselves here in a special situation where
the sequence can be defined more succinctly.

If $L$ is any algebra of type $1$ then we may choose a nonzero element
$y$ of $L_1$ such that $[L_2,y]=0$.
Following~\cite{CMN} we say that an algebra $L$ of type $1$ is {\em uncovered} if there is an element
$z\in L_1$ such that $L_i=[L_{i-1},z]$ for all $i \ge 1$.
Being uncovered is a mild restriction as it can always
be achieved by extending the base field $F$ of $L$
(and is automatically satisfied if $F$ is uncountable),
but allows a simpler description of $L$.
In fact, if $L$ is uncovered and we have made a choice of $z$,
then we may use that to recursively define a basis of $L$
(together with $z$ itself) by setting $e_2=[y,z]$ and
$e_i=[e_{i-1},z]$ for all $i>2$.
The {\em sequence} $(\alpha_i)_{i>1}$
of $L$ is then defined by the equation
$[e_i,y]=\alpha_{i}e_{i+1}$ for all $i>1$.
Hence the sequence of $L$ encodes the adjoint action of $y$ on $L$,
and completely determines $L$ up to isomorphism
(because $y$ and $z$ generate $L$ as an algebra, and
the adjoint action of $z$ is used to recursively define the elements $e_i$).
Beware that, despite our use of the definite article,
{\em the} sequence of $L$ actually depends on the choice
of $y$ and $z$,
but that will not concern us in this paper.

We briefly pause our discussion to introduce some notation
and a useful general tool.
Long Lie products are to be interpreted using the left-normed convention, meaning that $[a,b,c]$ stands for $[[a,b],c]$.
A convenient shorthand is $[a, c^i]$,
where $[a,c^0]=a$ and $[a,c^i]=[a,c^{i-1},c]$ for $i>0$.
We will make frequent use of the {\em generalized Jacobi identity}
\[
[a,[b,c^i]]=\sum_{j=0}^{i}(-1)^j\binom{i}{j}[a,c^j,b,c^{i-j}].
\]

The special choice of an element $y$ with $[L_2,y]=0$
in an algebra $L$ of type $1$ has the crucial consequence
that at least one of each pair of consecutive entries of the sequence vanishes~\cite[Lemma~3.3]{CMN}.
This fact is equivalent to $[L_i,y,y]=0$ for all $i\ge 1$.
When $i$ equals $1$ or $2$ this holds because $[L_2,y]=0$
by definition of $y$.
Proceeding inductively,
suppose $[L_i,y]\neq 0$ for some $i>2$,
and assume $[L_{i-1},y]=0$ as we may.
Then letting $u$
span $L_{i-1}$ we have that $[u,z]$ spans $L_i$, and
\begin{equation*}
0=[u,[z,y,y]]=[u,z,y,y]-2[u,y,z,y]+[u,y,y,z]=[u,z,y,y]
\end{equation*}
yields $[L_i,y,y]=0$ as desired.
Incidentally, in characteristic different from two
the condition $[L,y,y]=0$ means that
$y$ is a {\em sandwich element} of $L$.
Some implications of this for algebras of type $1$ and for
{\em thin Lie algebras} as well are discussed in~\cite{Mat:chain_lengths}.
For the present discussion, it means that the sequence of an algebra of type $1$ consists of isolated nonzero entries,
separated by blocks of one or more zeroes.

We now introduce algebras of type $n$.
Their name refers to the fact that they are $\N$-graded
Lie algebras of maximal class, generated by
an element of degree $1$ together with an element of degree $n$, but our definition will be more restrictive than that.
The broader definition and the further challenges it poses
are discussed in~\cite[Remark~5]{IMS}.
\begin{definition}
An {\em algebra of type $n$,} with $n>1$, is a graded Lie algebra $L=\bigoplus_{i=1}^{\infty}L_i$
with $\dim L_i=0$ for $1<i<n$ and $\dim L_i=1$ otherwise,
such that $[L_i,L_1]=L_{i+1}$ for all $i\ge n$.
\end{definition}
After choosing nonzero elements $z\in L_1$ and $e_n\in L_n$
we may model a definition of the sequence of $L$
on that for uncovered algebras of type $1$.
We recursively set $e_i=[e_{i-1},z]$ for $i>n$.
Hence $L_i=F e_i$ except for $L_1=F z$, and of course $L_i=0$ for $1<i<n$ by definition.
We define {\em the sequence} $(\beta_i)_{i>n}$ of $L$ by
$[e_i, e_n]=\beta_i e_{i+n}$.
This includes the case of uncovered algebras of type $1$ as a special case, except that $y$ was used for $e_1$ there, and $L_1=F z+F y$.
Note that our element $z$ was denoted by $e_1$
in~\cite{CVL00,CVL03}, but our present usage emphasizes its special
self-centralizing property.
Again, this sequence of elements of $F$ encodes the adjoint action
of $e_n$, while the action of $z$ is actually used to recursively
define the chosen basis elements $e_i$.
It follows that the Lie algebra $L$ is completely determined by this sequence, because the adjoint action of the generators
determines the adjoint action of every element of the Lie algebra.
Note that the equation used to define $\beta_i$ could be used to extend its definition to $\beta_n=0$,
but for various reasons
it is best to leave $\beta_n$ undefined and let the sequence begin with $\beta_{n+1}$.

The sequence $(\beta_i)_{i>n}$ is identically zero for the algebra of type $n$ having multiplication
$[e_i,e_n]=0$ for all $i>n$ (and $e_i:=[e_n,z^{i-n}]$ as postulated above), which has an abelian maximal ideal.
Differently from algebras of type $1$,
of which there is a unique metabelian one,
for type $n>1$ there is precisely one further metabelian algebra,
having an abelian ideal of codimension $2$ but no abelian maximal ideal.
Its multiplication may be taken as
$[e_i,e_n]=e_{i+n}$ for $i>n$, and $[e_i,e_j]=0$ for $i,j>n+1$,
making its sequence have all entries equal to $1$.

Aside from the algebra of type $n$ with an abelian maximal ideal,
the sequence of an algebra of type $n>1$ depends on the choice of
its generators $z$ and $e_n$,
but only up to multiplication by a nonzero scalar.
If one works with a \emph{normalized sequence}
whose first non-zero entry equals $1$,
then two algebras of type $n>1$ are isomorphic if and only if their normalized sequences coincide.

\section{Constituents and their lengths}\label{sec:constituents}

In this section we recall the notion of {\em constituents}
for uncovered algebras of type $1$, we appropriately extend it to algebras
of type $n$, and we state several of their properties.
Because all this material has received a detailed presentation
in~\cite[Sections~3--5]{IMS}, here we limit ourselves
to a concise exposition
and refer to that paper for further details and all the proofs.

Constituents and their lengths were crucial in the investigation of algebras of type $1$ started in~\cite{CMN},
and in their subsequent classification in~\cite{CN,Ju:maximal}.
Consider an uncovered algebra $L$ of type $1$, and make a choice
of $y$ and $z$ in $L_1$, as described in the previous section.
If $L$ is metabelian then its sequence vanishes identically.
Otherwise, we saw in the previous section that the sequence of $L$ consists of isolated occurrences of nonzero scalars $\alpha_j$,
separated by strings of one or more zeroes.
This suggests a natural partition of the sequence into juxtaposed blocks, each made of zeroes except for a final nonzero entry.
Those blocks (or finite subsequences) are called the {\em constituents} of $L$,
and will be numbered in order of occurrence (the {\em first} constituent, the {\em second,} etc.).
The {\em length} of a constituent is simply the number of its entries,
except for the first constituent, where that number needs to be increased by one.

For an algebra $L$ of type $n>1$
we also divide its sequence $(\beta_i)_{i>n}$
into adjacent blocks, and call those its {\em constituents,}
but the appropriate definition is less obvious, because
the nonzero entries of the sequence are usually not isolated.
Assume the sequence of $L$ is not the zero sequence.

\begin{definition}\label{def:first_constituent}
The {\em first constituent} of $L$ is the sequence
$(\beta_{n+1},\ldots,\beta_{\ell})$,
where $\beta_{\ell-n+1}$ is the first nonzero entry of the sequence $(\beta_i)_{i>n}$.
\end{definition}

Thus, $\ell$ is determined by $[e_{\ell-n+1},e_n]\neq 0$ but $[e_i, e_n]=0$ for $n\le i\le\ell-n$.
We stipulate that the first constituent has length $\ell$ (despite actually having only $\ell-n$ entries as a sequence).
We define further constituents recursively.

\begin{definition}\label{def:constituent}
If $(\beta_i,\ldots,\beta_j)$ is a constituent which we have already defined,
and $\beta_{j+m-n+1}$ is the first nonzero entry
past $\beta_j$
(which exists according to Lemma~\ref{lemma:upper_bound} below), then
we define $(\beta_{j+1},\ldots,\beta_{j+m})$ to be the next constituent, of length $m$.
\end{definition}

Thus, the sequence of $L$ gets partitioned into a sequence of adjacent constituents.
The above definitions are consistent with those given in~\cite{CVL03} for algebras of type $2$.
Note that, by definition, every constituent has length at least $n$, and the first constituent has length at least $2n$.
We denote the constituent lengths by $\ell=\ell_1,\ell_2,\ell_3,\ldots$, numbered in order of occurrence.

Some guidance for the above definitions of constituents originates
from the very special setting of an algebra of type $n$
arising as a graded subalgebra of an uncovered algebra of type $1$,
which we recall now.
Let $N=\bigoplus_{i \ge 1} N_i$ be an uncovered algebra of type $1$,
make a choice of $z$ and $e_1$, and set $e_i=[e_1,z^{i-1}]$ as usual,
for $i>1$.
Then fix $n>1$ and consider
the (graded) subalgebra $L$ of $N$  generated by $z$ and $e_n$.
Note that different choices of $z$ in $N$
may lead to non-isomorphic subalgebras $L$.
Then $L=L_1\oplus \bigoplus_{i \ge n} L_i$ is an algebra of type $n$,
where $L_1=F z$, and $L_i=N_i$ for $i\ge n$.
By computing $[e_n,e_i]$ within $N$ (see~\cite{IMS} for details)
one finds that the sequence $(\beta_i)_{i>n}$ of $L$ can be obtained
from the sequence $(\alpha_j)_{j>1}$ of $N$ according to the equation
$\beta_i=\sum_{k=0}^{n-1} (-1)^k \binom{n-1}{k}\alpha_{i+k}$.
If every constituent of $N$ has length at least $n$,
whence consecutive nonzero entries $\alpha_j$ of the sequence of $N$
are spaced at least $n$ apart,
the contribution to each $\beta_i$ in the above equation comes
from a single $\alpha_j$.
It is then easy to see that all constituents of $L$ satisfy
the following definition.

\begin{definition}
Any constituent
$(\lambda_{j-m+1},\ldots,\lambda_j)$ of $L$ having the form
\[
\lambda_{j-i}
=
(-1)^i\binom{n-1}{i}\lambda_j
\qquad\text{for $0\le i<m$,}
\]
is called {\em ordinary,}
and {\em ending in $\lambda_j$} if we need to specify that.
\end{definition}

The following result,
which extends~\cite[Lemma~3.1.2]{Ugo:thesis},
provides the converse implication.

\begin{prop}[\cite{IMS}, Proposition~6]\label{OrdinaryEnding}
If $L$ is an algebra of type $n$
and all constituents of $L$ are ordinary,
then $L$ is isomorphic to a graded subalgebra
of an uncovered algebra of type $1$.
\end{prop}

While ordinary constituents necessarily have a nonzero last entry,
this is generally not the case for arbitrary constituents,
prompting the following definition.

\begin{definition}\label{def:leading_trailing}
The \emph{leading term} of a constituent is its first nonzero entry,
and the \emph{trailing term} is its last nonzero entry.
When we say that $\beta_t$ is the
leading or trailing term of a constituent, we mean this information to
include the specific index $t$ in the sequence, rather than just the
particular (nonzero) value $\beta_t$ takes in the base field.
\end{definition}

Thus, by definition the leading term of a constituent is precisely the $n$th
term counting from the end.
Consequently, if $(\beta_{i},\ldots,\beta_{j})$ and $(\beta_{j+1},\ldots,\beta_{j+m})$
are two adjacent constituents, then the length $m$ of the latter one equals
the difference between the indices of the respective leading terms
$\beta_{j-n+1}$ and $\beta_{j+m-n+1}$.
In contrast, the position of the trailing term of a constituent
within it depends on the situation.
The importance of trailing terms lies in the fact that many arguments
exploit a maximal subsequence of consecutive zero entries
in the sequence of the algebra.
Some such sequences can be found between the trailing term of
a constituent and the leading term of the following constituent.

We now quote from~\cite[Section 4]{IMS} some general facts on the constituent lengths $\ell=\ell_1,\ell_2,\ell_3,\ldots$ of an algebra $L$ of arbitrary type $n$.
They extend corresponding facts proved in~\cite{CMN} for algebras of type $1$.

\begin{lemma}[\cite{IMS}, Lemma~8]\label{lemma:ell_even}
The length $\ell$ of the first constituent of a non-metabelian algebra $L$ of type $n$ is even.
\end{lemma}

\begin{lemma}[\cite{IMS}, Lemma~9]\label{lemma:upper_bound}
Let $L$ be a non-metabelian algebra of type $n$, with first constituent of length $\ell$.
Then its sequence $(\beta_i)_{i>n}$ can contain at most $\ell-n$ consecutive zero entries.
Consequently, the constituents lengths $\ell=\ell_1,\ell_2,\ell_3,\ldots$ of $L$ satisfy $\ell_r\le\ell$ for all $r$.
\end{lemma}

The upper bound on $\ell_r$ guarantees, in particular, that all constituent lengths are finite,
which we assumed in Definition~\ref{def:constituent}.

\begin{lemma}[\cite{IMS}, Lemma~10]\label{lemma:constituent_bound}
Let $L$ be a non-metabelian algebra of type $n$, with constituent lengths $\ell=\ell_1,\ell_2,\ell_3,\ldots$ in order of occurrence.
Then $\ell_r\ge \ell/2$ for all $r$.
\end{lemma}

Finally, we quote from~\cite[Section~5]{IMS} an internal
characterization of the constituents of an algebra $L$
of type $n$, which avoids reference to the sequence $(\beta_i)_{i>n}$,
and which we have used in the Introduction to describe them more succinctly.

\begin{prop}[\cite{IMS}, Proposition~11]\label{prop:constituent_lengths}
Let $L$ be an algebra of type $n>1$ such that $[L^2,L_n]\subseteq L^{n+3}$.
Then the constituent lengths of $L$ are given by
$\ell=\ell_1=\dim\bigl(L^2/(L^2)^2\bigr)+n$
and
$\ell_r=\dim\bigl((L^2)^r/(L^2)^{r+1}\bigr)$ for $r>1$.
\end{prop}

The condition
$[L^2,L_n]\subseteq L^{n+3}$
is necessary for Proposition~\ref{prop:constituent_lengths} to hold,
and is equivalent to requiring
$\beta_{n+1}=0$.
In turn, the latter is equivalent to
$\ell$ exceeding its smallest allowed value $2n$.
Because $\ell$ is even according to Lemma~\ref{lemma:ell_even},
the condition actually implies $[L^2,L_n]\subseteq L^{n+4}$.

\section{Polynomial arguments}~\label{sec:polynomials}

In this section we present a result on polynomials
which might be of independent interest,
but will be crucial to prove Theorem~\ref{thm:first_length},
and also provides some intuition for those particular
values of $\ell$ which it allows.

In fact, the main argument in our proof of Theorem~\ref{thm:first_length},
in Subsection~\ref{subsec:main},
will show the vanishing of a certain range of coefficients in the product
$(x-1)^{\ell-n+1}g(x)$,
where $g(x)$ is a polynomial of degree $n-1$ encoding the last $n$ entries of the first constituent of an algebra of type $n$.
Theorem~\ref{thm:polynomials} below shows that
such condition on $(x-1)^{\ell-n+1}g(x)$
restricts the possibilities for $\ell$ to those in the conclusion
of Theorem~\ref{thm:first_length} and a few more,
the latter to be excluded later by different calculations in $L$.

In broad terms, writing $k$ in place of $\ell-n+1$,
and letting $g(x)$ a polynomial of degree $n-1$
over a field $F$ of positive characteristic $p$,
those conditions on $(x-1)^kg(x)$ say
that roughly the upper half of its coefficients
vanish, necessarily excluding the leading term.
Following standard notation we denote by $[x^j] f(x)$ the coefficient of $x^j$ in a polynomial $f(x)$.
To gain some intuition on Theorem~\ref{thm:polynomials}
consider $n=1$ (not included there) as a simple model.
Then the hypothesis of Theorem~\ref{thm:polynomials} would read
$[x^j](x-1)^k=0$ for $(k+1)/2\le j<k$,
and would easily imply that $k$ equals either $q$ or $2q$, for some power
$q$ of $p$.
In fact, because of the binomial identity $\binom{k}{k-j}=\binom{k}{j}$
the assumed condition would extend to the range $0<j<k$ when $k$ is odd,
and the same range with the exception of $j=k/2$ when $k$ is even.
In the former case we would find $(x-1)^k=x^k-1$,
and then there are several ways to deduce $k=q$.
In the latter case we would get $(x-1)^k=x^k+ax^{k/2}+1$,
and then it is not much harder to infer that either $k=q$ or $k=2q$,
for example as in~\cite[Lemma~1]{Mat:chain_lengths}.

The above basic facts can also be proved using Lucas' theorem,
a useful  tool for evaluating a binomial coefficient
$\binom{a}{b}$ modulo a prime $p$:
if $a$ and $b$ are non-negative integers with $p$-adic expansions $a=a_0+a_1p+\cdots+a_rp^r$ and $b=b_0+b_1p+\cdots+b_rp^r$, then
\[
\binom{a}{b}=\prod_{i=0}^{r}\binom{a_i}{b_i} \pmod{p}.
\]

The following preparatory result is akin to the special case $n=2$
of Theorem~\ref{thm:polynomials}, but not quite the same.
It comes in two variants, differing only when $k$ is odd,
both of which will be needed in the proof of Theorem~\ref{thm:polynomials}.

\begin{lemma}[\cite{IMS}, Lemma~16]\label{lemma:polynomials}
Let $F$ be a field of characteristic $p$ and let $x-a \in F[x]$. Suppose that, for a natural number $k>1$, we have
\begin{equation}\label{eq:range_lemma}
[x^j](x-1)^k(x-a)=0 \quad\text{for } k/2+1\le j\le k.
\end{equation}
Then the pair $(k,a)$ is one of the following: either $(2,-2)$, or $(3,-3)$, or $(q-1,1)$, or $(q,0)$, or $(2q-1,1)$, where $q$ is a power of $p$.

Moreover, if Equation~\eqref{eq:range_lemma} holds in the
range $(k+1)/2\le j\le k$, then the pairs $(3,-3)$ and $(2q-1,1)$ can be excluded from the conclusion.
\end{lemma}

The special case of Lemma~\ref{lemma:polynomials}
where $a=1$ recovers the basic facts on $(x-1)^{k+1}$
recalled in the above discussion
(which was expressed in terms of $(x-1)^k$ there).

\begin{thm}\label{thm:polynomials}
Let $F$ be a field of positive characteristic $p$,
and let $n$ and $k$ be integers with $1<n<p$ and $k>n+1$.
Suppose there is a monic polynomial $g(x)\in F[x]$, of degree $n-1$,
such that
\begin{equation}\label{eq:range}
[x^j](x-1)^kg(x)=0 \quad\text{for } (k+n)/2\le j<k.
\end{equation}
Then the following assertions hold.
\begin{itemize}
\item[(a)]
If $k\ge 4p$, then
$k=2q-n+1$
or $q-n<k<q+n$,
where $q>p$ is a power of $p$.
\item[(b)]
If $k<4p$, then k belongs to one of the following intervals:
$n+1<k<p$,
$2p-n<k<2p$,
$3p-n<k<3p$,
$k=4p-n+1$.

\item[(c)]
If $k=2q-n+1$,
then $g(x)=(x-1)^{n-1}$.
\item[(d)]
If $k=q-p+k_0$, with $p-n<k_0<p$, then $(x-1)^{p-k_0}$ divides $g(x)$.
\item[(e)]
If $k=q+k_0$ with $0<k_0<n$, then $x^{k_0}$ divides $g(x)$.
\end{itemize}
\end{thm}

Theorem~\ref{thm:polynomials} is similar to~\cite[Theorem~17]{IMS},
which deals with the case $n=p$ not covered here.
However, the conclusions are slightly more complicated here,
and the differences with the proof of~\cite[Theorem~17]{IMS}
are significant enough to justify writing out a full proof of
Theorem~\ref{thm:polynomials}.
Among all the possibilities for $k$ in the conclusions we have singled out
the typical ones in Assertion~(a), with those of Assertion~(b)
representing exceptional small values.
The hypothesis $\ell>4p$ in Theorem~\ref{thm:first_length} will serve
precisely to exclude the values in Assertion~(b),
when applying the above polynomial result in its proof.

Our hypothesis $k>n+1$ in Theorem~\ref{thm:polynomials} is due to the condition expressed in Equation~\eqref{eq:range} being void otherwise.
Note that the assumption on $n$ in Theorem~\ref{thm:polynomials}
implies $p>2$.

\begin{proof}[Proof of Theorem~\ref{thm:polynomials}]
Write $k=k'p+k_0$ with $0\le k_0<p$,
whence we may write
\begin{equation*}
(x-1)^k g(x)= (x^p-1)^{k'} (x-1)^{k_0} g(x).
\end{equation*}
Our proof is broadly similar to the proof of~\cite[Theorem~17]{IMS},
except that here we distinguish two cases,
depending on whether $k_0\le p-n$, or not.
The former case is an extension of the simplest case in the proof
of~\cite[Theorem~17]{IMS},
which was the case $k_0=0$ there.
The differences with the proof of~\cite[Theorem~17]{IMS}
are significant enough to justify spelling out our  arguments in full.

Thus, write $g(x)=\sum_{i=0}^{n-1}g_ix^i$, whence $g_{n-1}=1$.
It will be convenient to define $g_i$ as zero when $i$ does not belong
to the range $0\le i<n$.
For a polynomial $f(x)\in F[x]$, and an integer $i$,
denote by $S_i\bigl(f(x)\bigr)$ the polynomial obtained from $f(x)$ by discarding all terms
where $x$ appears with an exponent not congruent to $i$ modulo $p$.
Consequently, $f(x)=\sum_{i=0}^{p-1}S_i\bigl(f(x)\bigr)$.

If $k_0\le p-n$, then
\begin{equation*}
    S_{k_0+n-1} \bigl( (x-1)^{k}g(x) \bigr)=(x^p-1)^{k'} x^{k_0+n-1}.
\end{equation*}
Equation~\eqref{eq:range} then implies
\begin{equation*}
[x^j](x^p-1)^{k'} x^{k_0+n-1}=0 \quad\text{for } (k+n)/2\le j<k,
\end{equation*}
which means
\begin{equation*}
[x^j](x^p-1)^{k'} =0 \quad\text{for } \frac{k' p}{2} - \frac{k_0+n-2}{2} \le  j < k' p + 1 - n.
\end{equation*}
Because $1<n<p$ this implies
\begin{equation*}
[x^j](x-1)^{k'} =0 \quad\text{for } \frac{k'}{2} \le j < k'.
\end{equation*}
As we recalled near the beginning of this section,
this condition forces $k'$ to be a power of $p$
(including possibly $p^0=1$).
We conclude $q\le k\le q+p-n$, with $q>1$ a power of $p$.

Consider now the case $k_0 > p-n$.
Because $p \le k_0 + n -1 < 2p$,
the only powers of $x$ in the product $(x-1)^{k_0} g(x)$
whose exponents are congruent to $k_0+n-1$ modulo $p$
are $x^{k_0+n-1}$ and $x^{k_0+n-1-p}$. Hence
\[
S_{k_0+n-1} \bigl((x-1)^{k_0} g(x)\bigr) = x^{k_0+n-1} - a x^{k_0+n-1-p},
\]
where
\[
a=-\sum_{s}(-1)^{k_0-s} \binom{k_0}{s}g_{k_0+n-p-1-s}.
\]
Consequently, we find
\begin{equation*}
S_{k_0+n-1}\bigl((x-1)^k g(x)\bigr)
=
(x^p-1)^{k'} (x^p-a)x^{k_0+n-p-1}.
\end{equation*}
Equation~\eqref{eq:range} then implies
$[x^j] (x^p-1)^{k'} (x^p-a)=0$
for
\[
\frac{k+n}{2} - k_0-n+p+1 \le j < k - k_0-n+p+1,
\]
which is equivalent to
$[x^j](x-1)^{k'} (x-a) =0$ for
\begin{equation*}
\frac{k'}{2}+1- \frac{k_0+n-2}{2p} \le j < k'+1-\frac{n-1}{p}.
\end{equation*}
Because $1<n<p$ we conclude
\begin{equation*}
[x^j](x-1)^{k'} (x-a) =0 \quad\text{for } \frac{k'}{2}+1 \le j \leq k', \end{equation*}
and even
\begin{equation*}
[x^j](x-1)^{k'} (x-a) =0 \quad\text{for } \frac{k'+1}{2} \le j \leq k' \end{equation*}
in case $k_0>p-n+1$.
Lemma~\ref{lemma:polynomials} then tells us that $k'$, if positive, can only equal $1$, $2$, $3$, $q'-1$, $q'$, or $2q'-1$, where $q'>1$ is a power of $p$.
Moreover, $k'$ cannot equal $3$ or $2q'-1$ if $k_0>p-n+1$.
Consequently, $k$ belongs to one of the following intervals, where $q>1$ is a power of $p$:
\[
n+1<k<p,
\qquad
2p-n<k<2p,
\qquad
3p-n<k<3p,
\qquad
k=4p-n+1,
\]
\[
q-n<k<q,
\qquad
q+p-n<k<q+p,
\qquad
k=2q-n+1.
\]

Together with the intervals $q\le k\le q+p-n$ found earlier
when $k_0\le p-n$,
these give the desired conclusions for $k$ stated in Assertions~(a) and~(b),
except for allowing the
unwanted intervals $q+n\le k<q+p$, for $q>p$, which we exclude next.
If $k=q+k_0$ with $0<k_0<p$,
we may write $(x-1)^k=(x^{q}-1)(x-1)^{k_0}$.
Because
Equation~\eqref{eq:range} holds, in particular, for $q\le j<q+k_0$,
we deduce
\[
[x^j](x-1)^{k_0} g(x)=0 \quad\text{for } 0 \le j< k_0.
\]
Because $x^{k_0}$ and $(x-1)^{k_0}$ are coprime it follows that $x^{k_0}$ divides $g(x)$.
Because $g(x)$ has degree $n-1$, this gives a contradiction
if $q+n\le k<q+p$.
This proves Assertion~(e) as well.

It remains to prove Assertions~(c) and~(d), which provide additional
information when $q-n<k<q$ or $k=2q-n+1$, with $q>p$.
Thus, assume $k=q-p+k_0$ with $p-n<k_0<p$.
Here Equation~\eqref{eq:range} reads
\begin{equation*}
[x^j] (x-1)^{q-p}(x-1)^{k_0}g(x)=0 \quad\text{for }
(q-p+k_0+n)/2\le j<q-p+k_0.
\end{equation*}
Now we expand $(x-1)^{q-p}=(x^q-1)/(x^p-1)$ as $x^{q-p}+x^{q-2p}+\cdots+x^p+1$,
and then restrict the range of $j$ to an interval where at most
the first two terms of this expansion matter.
Note that under our assumptions on $q$, $n$ and $k_0$ we have
$q-2p+k_0\ge (q-p+k_0+n)/2$.
In fact, this means $q+k_0\ge 3p+n$,
which holds when $p>3$ because $q\ge p^2\ge 5p$,
but also when $p=3$ because $n=2$ and $k_0=2$ in that case.
Hence we infer
\begin{equation*}
[x^j] x^{q-2p}(x^p+1)(x-1)^{k_0}g(x)=0 \quad\text{for } q-2p+k_0\le j < q-p+k_0,
\end{equation*}
which amounts to
\[
[x^j] (x^p+1)(x-1)^{k_0}g(x)=0 \quad\text{for } k_0\le j < p+k_0.
\]
If we further restrict the range of $j$ in the above equation,
looking only at terms of degree less than $p$, we find
\begin{equation*}
[x^j] (x-1)^{k_0}g(x)=0 \quad\text{for } k_0\le j < p.
\end{equation*}
Hence $(x-1)^{k_0}g(x)=x^pa(x)+b(x)$, with $a(x)$ and $b(x)$ polynomials of degree less than $k_0$.
Because $1$ is a root of $x^pa(x)+b(x)$ with multiplicity at least $k_0$, the derivatives
of this polynomial up to order $k_0-1$ also have $1$ as a root.
Since $x^p$ has zero derivative the analogous conclusion holds for the corresponding derivatives of $a(x)+b(x)$.
Because $k_0<p$ this implies that $a(x)+b(x)$ has $1$ as a root with multiplicity at least $k_0$, which exceeds its degree,
whence $a(x)+b(x)$ is the zero polynomial.
Consequently,
$(x-1)^{k_0}g(x)=x^pa(x)-a(x)=(x-1)^p a(x)$,
and hence $(x-1)^{p-k_0}$ divides $g(x)$, as desired.
This proves Assertion~(d).

Finally, to prove Assertion~(c),
when $k=2q-n+1$ Equation~\eqref{eq:range} reads
\begin{equation*}
[x^j] (x-1)^{2q-p}(x-1)^{p-n+1}g(x)=0
\quad\text{for }
q<j<2q-n+1.
\end{equation*}
Because $(x-1)^{2q-p}=(x^q-1)\sum_{i=0}^{q/p-1}x^{ip}$,
its restriction to the range $q<j<q+p$ implies
\begin{equation*}
[x^j] (x-1)^{p-n+1}g(x)=0
\quad\text{for }
0<j<p,
\end{equation*}
whence
$(x-1)^{p-n+1}g(x)=x^p-b$
for some $b\in F$.
Because $n<p$ this polynomial has $1$ as a root, hence $b=1$,
and the desired conclusion $g(x)=(x-1)^{n-1}$ follows.
\end{proof}

\section{Proof of Theorem~\ref{thm:first_length}}\label{sec:first_length_proof}

Let $L$ be an algebra of type $n$, hence generated by $z$ and $e_n$ with the usual meaning, and with first constituent length $\ell>4p$. Our main argument, in Subsection~\ref{subsec:main}, will use information on the first and second constituent,
 as made explicit in Equations~\eqref{eq:first_const}--\eqref{eq:second_const_end},
together with Lemma~\ref{lemma:constituent_bound},
to show that the last $n$ two-step centralizers of the first constituent satisfy a linear recurrence (in their range).
This can be interpreted as a condition on the coefficients of a certain polynomial which will allow an application of Theorem~\ref{thm:polynomials}.
However, the conclusions of this argument will allow for some values for $\ell$, namely $q+n <\ell \le q+2n-3$,
which are not genuine possibilities, and require further calculations to exclude.
We do that in the second part of the proof, which begins with Subsection~\ref{subsec:spurious} and is considerably more complicated.
In fact, excluding certain configurations, described in
Theorem~\ref{thm:final_contradiction}, will require
the whole Section~\ref{sec:final_contradiction}.

\subsection{The main argument}\label{subsec:main}
For convenience we state Equations~\eqref{eq:first_const}--\eqref{eq:second_const_end},
which encode the first and second constituent of $L$,
where we have conveniently normalized $\beta_{\ell-n+1}=1$:
\begin{subequations}
\begin{align}
[e_i, e_n]&=0 \qquad\text{for }n<i\le\ell-n=k-1,
 \label{eq:first_const}\\
[e_k, e_n]=[e_{\ell-n+1},e_n] &=e_{\ell+1}=e_{k+n},
 \label{eq:first_const_end}\\
[e_{k+n-1+j},e_n]=[e_{\ell+j},e_n]&=0\qquad\text{ for }0<j\le \ell_2-n,
 \label{eq:second_const}\\
[e_{k+\ell_2},e_n]=[e_{\ell+\ell_2-n+1},e_n]&=\beta_{\ell+\ell_2-n+1}e_{\ell+\ell_2+1}\ne 0.
 \label{eq:second_const_end}
\end{align}
\end{subequations}

Note that at this stage we have no information on
$[e_i, e_n]$ for $\ell-n+1<i\le\ell$,
as we do not know the corresponding entries $\beta_{\ell-n+2},\ldots,\beta_{\ell}$ of the first constituent.
We will use Equations~\eqref{eq:first_const}--\eqref{eq:second_const_end}
to obtain information on the adjoint action of $e_{\ell+1}$.

Our arguments will frequently involve
expanding certain Lie brackets $[e_a, e_b]$ with $a, b \ge n$.
It is then convenient to note once and for all that
\begin{align*}
[e_a, e_b] = \sum_{i} (-1)^i \binom{b-n}{i} \beta_{a+i} e_{a+b}.
\end{align*}
This is because
$[e_a, e_b] = [e_a, [e_n, z^{b-n}]]$ and
\begin{align*}
[e_a, [e_n, z^{b-n}]] & = \sum_{i} (-1)^i \binom{b-n}{i} [e_{a+i}, e_n, z^{b-n-i}]\\
& = \sum_{i} (-1)^i \binom{b-n}{i} \beta_{a+i} e_{a+b}.
\end{align*}
As we have just done here, we will also conveniently
omit the explicit summation range
$0\le i\le b-n$ because the binomial coefficient involved
vanishes outside that range.

Our first observation is that $e_{\ell+1}$ centralizes $e_{j}$ for $n\le j<\ell_2$, namely,
\[
[e_{\ell+1},e_{j}]
=\sum_{i} (-1)^i\ \binom{j-n}{i} \beta_{\ell+1+i} e_{\ell+1+j}
=0,
\]
because $\beta_{\ell+1+i}=0$ for $0\le i<\ell_2-n$.
Consequently, after interchanging the entries of $[e_{\ell+1},e_{j}]$ and expanding by the generalized Jacobi identity we find
\begin{align*}
0=[e_{j},e_{\ell+1}]
=\sum_{i} (-1)^i \binom{\ell+1-n}{i} \beta_{j+i} e_{j+\ell+1}.
\end{align*}
Now note that $j\le j+i\le j+\ell-n+1\le\ell+\ell_2-n$.
Hence $\beta_{j+i} = 0$ in the above summation except possibly for
$\ell-n<j+i\le\ell$,
that is,
$\ell-n-j<i\le\ell-j$.
After an appropriate change of summation index, we conclude
\begin{equation*}\label{linear recurrence}
\sum_{i} (-1)^{j+i} \binom{\ell-n+1}{\ell-j-i} \beta_{\ell-i}=0
\qquad \text{for } n\le j<\ell_2.
\end{equation*}

The left-hand side of the above equation
equals the coefficient of $x^{\ell-j}$ in the product $(1-x)^{\ell-n+1}g(x)$, where
\[
g(x)=\beta_{\ell-n+1}x^{n-1}+\beta_{\ell-n+2}x^{n-2}+\cdots+\beta_{\ell},
\]
a polynomial of degree $n-1$, which is actually monic due to our normalization $\beta_{\ell-n+1}=1$.
Thus, we have found that $g(x)$ satisfies
\begin{equation}\label{fromLinRecurToZeroCoef}
[x^j] (x-1)^{\ell-n+1} g(x)=0 \qquad \text{for } \ell-\ell_2<j\le\ell-n.
\end{equation}
Because $\ell_2\ge \ell/2$ according to Lemma~\ref{lemma:constituent_bound}, and setting $k=\ell-n+1$, this implies
\begin{equation*}
[x^j] (x-1)^k g(x)=0 \text{\quad for } (k+n+1)/2\le j<k.
\end{equation*}
Because $k+n$ is odd according to Lemma~\ref{lemma:ell_even}, the range is equivalent to $(k+n)/2\le j<k$,
and so this condition is equivalent to Equation~\eqref{eq:range} in Theorem~\ref{thm:polynomials}.
In addition, here $k>4p-n+1$ because of our hypothesis $\ell>4p$.
Thus, Theorem~\ref{thm:polynomials} yields that
either $k=2q-n+1$ or $q-n<k<q+n$, where $q>p$ is a power of $p$.
Thus, in terms of $\ell=k+n-1$ we conclude that either $\ell=2q$ or $q\le\ell\le q+2n-2$.
Because $\ell$ must be even according to
Lemma~\ref{lemma:ell_even},
these inequalities effectively mean
$q<\ell< q+2n-2$.

\subsection{Excluding the spurious possibilities for $\ell$}\label{subsec:spurious}
In order to complete a proof of Theorem~\ref{thm:first_length} it remains to rule out the possibilities $q+n+1 \le \ell\le q+2n-3$
for the length of the first constituent,
which means $q+2 \le k\le q+n-2$ in terms of $k$.

Thus, assume $k=q+k_0$, where $q>p$ is a power of $p$ and $k_0$ is an integer with $2 \le k_0 \le n-2$.
This easily implies $\ell_2\le q$. In fact, for $\ell_2 > q$ we get
\[
\ell - \ell_2 = q + k_0 + n-1 - \ell_2 < k_0 + n -1.
\]
Since $(x-1)^k=(x^q-1)(x-1)^{k_0}$ we have
\[
[x^{k_0+n-1}](x-1)^k g(x)=-[x^{k_0+n-1}](x-1)^{k_0}g(x)=-1,
\]
in contradiction with the fact that $[x^{k_0+n-1}](x-1)^k g(x) = 0$, according to Equation~\eqref{fromLinRecurToZeroCoef}.

With some harder work we will show that $\ell_2=q$ or $q-1$,
and then show how either possibility leads to a contradiction.

Theorem~\ref{thm:polynomials} provides further information on $g(x)$, namely, that it is a multiple of $x^{k_0}$.
That means
\begin{equation*}
\beta_{q+n}=\cdots=\beta_{q+k_0+n-1}=0,
\end{equation*}
or, in words, that the last $k_0$ entries of the first constituent vanish.
Put differently, under our assumption on $\ell$, Equation~\eqref{eq:second_const} holds over an extended range, namely,
\begin{equation}\label{eq:second_const_plus}
[e_{\ell+j},e_n]=0\quad\text{ for }-k_0<j\le \ell_2-n.
\end{equation}
In the rest of the proof we need to refine this partial information on $g(x)$.
Recall that the leading term of the second constituent is $\beta_{q+k_0+\ell_2}$,
and consider its trailing term $\beta_s$,
hence $k_0\le s<n$ and $s$ is largest such that $\beta_{q+s}\neq 0$.
In terms of $g(x)$ this means that
the highest power of $x$ dividing $g(x)$ is $x^{k_0+n-1-s}$.

Because of our hypothesis $\ell>4p$ and $p > n$, and because we know $\ell_2\ge\ell/2$ from Lemma~\ref{lemma:constituent_bound},
we have $\ell_2+k_0 > 2p > s+n$. Hence we may write $e_{\ell_2+k_0-s} = [e_n, z^{\ell_2+k_0-s-n}]$ and compute
\begin{align*}
[e_{\ell_2+k_0-s},e_{q+s}]
&=-[e_{q+s},e_{\ell_2+k_0-s}] \\
& = -\sum_{j} (-1)^j \binom{\ell_2+k_0-s-n}{j} \beta_{q+s} e_{q+k_0+\ell_2} = -\beta_{q+s} e_{q+k_0+\ell_2},
\end{align*}
where all remaining terms in the summation vanish because of our choice of $s$.

Consequently,
\begin{align*}
\beta_{q+s}[e_{\ell_2+k_0-s},e_{q+s+n}]
&=
[e_{\ell_2+k_0-s},[e_{q+s},e_n]]
\\&=
[e_{\ell_2+k_0-s},e_{q+s},e_n]-[e_{\ell_2+k_0-s},e_n,e_{q+s}]
\\&=
[e_{\ell_2+k_0-s},e_{q+s},e_n]
\\&=
-\beta_{q+s}[e_{q+k_0+\ell_2},e_n],
\end{align*}
where the second term in the second line vanishes because
$\ell_2+k_0-s\le q<q+k_0$,
due to $\ell_2\le q$ as we proved earlier.
Because $\beta_{q+s}\neq 0$ we deduce
\begin{equation}\label{eq:s_even}
[e_{\ell_2+k_0-s},e_{q+s+n}]
=-[e_{q+k_0+\ell_2},e_n] = - [e_{k+\ell_2}, e_n] \neq 0.
\end{equation}
However, we also have
\begin{align*}
[e_{\ell_2+k_0-s},e_{q+s+n}]
&=
\sum_{j} (-1)^j \binom{q+s}{j}[e_{\ell_2+k_0-s+j},e_n,z^{q+s-j}].
\end{align*}
The binomial coefficients in the above summation vanish modulo $p$ except for the terms with $0 \le j \le s$ or $q\le j\le q+s$.
Because of the length of the second constituent all terms in the latter range vanish except when $j=q+s$. In fact
\[
q + n < \ell_2+k_0-s+j < \ell_2 + k_0 + q \quad \text{for $q \le j < q+s$.}
\]
Because $\ell_2\le q$, in the former range we have $\ell_2+k_0-s+j \le q+k_0$, whence the corresponding Lie product vanishes,
except possibly for $j=s$ in case $\ell_2=q$.
Thus,
\[
[e_{\ell_2+k_0-s},e_{q+s+n}]
=
(-1)^s[e_{\ell_2+k_0},e_n,z^{q}]
+(-1)^{q+s}[e_{q+\ell_2+k_0},e_n].
\]
Combined with Equation~\eqref{eq:s_even}, if $s$ is odd this yields
\[
[e_{\ell_2+k_0},e_n,z^{q}]
=
2[e_{q+\ell_2+k_0},e_n]\neq 0,
\]
whence $\ell_2=q$ as observed above.

To deal with the case where $s$ is even we need a slightly different calculation, in degree one higher.
Similar calculations as in the previous case yield
\[
[e_{\ell_2+k_0-s+1},e_{q+s}]
=-[e_{q+s},e_{\ell_2+k_0-s+1}]
=-\beta_{q+s} e_{q+\ell_2+k_0+1},
\]
and, since $\ell_2 + k_0 - s + 1 \leq q + 1 < q + k_0$, also
\[
\beta_{q+s}[e_{\ell_2+k_0-s+1},e_{q+s+n}]
=
[e_{\ell_2+k_0-s+1},e_{q+s},e_n]
=
-\beta_{q+s}[e_{q+\ell_2+k_0+1},e_n],
\]
whence
\begin{equation}\label{eq:s_odd}
[e_{\ell_2+k_0-s+1},e_{q+s+n}]
=-[e_{q+\ell_2+k_0+1},e_n].
\end{equation}
We also have
\begin{align*}
[e_{\ell_2+k_0-s+1},e_{q+s+n}]
&=
\sum_{j} (-1)^j \binom{q+s}{j}[e_{\ell_2+k_0-s+1+j},e_n,z^{q+s-j}]
\end{align*}
As in the previous case the binomial coefficients in the above summation vanish modulo $p$ except for the terms with $0\le j\le s$ or $q\le j\le q+s$.
However, here Lie products in the latter range vanish except when $j=q+s-1$ and possibly $q+s$.
Also, Lie products in the former range vanish as long as $\ell_2+k_0-s+1+j<q+k_0$.
This holds with the possible exceptions of $j=s-1$ or $s$, which may only arise when $\ell_2=q$ (in both cases) or $q-1$ (in the latter case).
Thus,
\begin{align*}
[e_{\ell_2+k_0-s+1},e_{q+s+n}]
&=
-(-1)^{s}s[e_{\ell_2+k_0},e_n,z^{q+1}]
+(-1)^s[e_{\ell_2+k_0+1},e_n,z^{q}]
\\&\quad
+(-1)^{s}s[e_{q+\ell_2+k_0},e_n,z]
-(-1)^{s}[e_{q+\ell_2+k_0+1},e_n]
\\&=
-s[e_{\ell_2+k_0},e_n,z^{q+1}]
+[e_{\ell_2+k_0+1},e_n,z^{q}]
\\&\quad
+s[e_{q+\ell_2+k_0},e_n,z]
-[e_{q+\ell_2+k_0+1},e_n]
\end{align*}

Combined with Equation~\eqref{eq:s_odd},  this yields
\[
0\neq s[e_{q+\ell_2+k_0},e_n,z]
=
s[e_{\ell_2+k_0},e_n,z^{q+1}]
-[e_{\ell_2+k_0+1},e_n,z^{q}],
\]
whence at least one of the terms at the right-hand side is nonzero,
implying $\ell_2=q$ or $\ell_2=q-1$, as claimed.

To complete a proof of Theorem~\ref{thm:first_length}
it remains to prove the following statement.

\begin{thm}\label{thm:final_contradiction}
There is no algebra of type $n$ such that $k = q+k_0$,
with $2 \leq k_0 \leq n-2$, and $\ell_2 = q$ or $\ell_2 = q-1$.
\end{thm}

Because the proof of Theorem~\ref{thm:final_contradiction}
is quite intricate, we postpone it to Section~\ref{sec:final_contradiction}.

\section{Construction of the exceptional algebras}\label{sec:exceptional}

In this section we construct the algebras of the exceptional family $\mathcal{E}$ in Theorem~\ref{thm:classification},
whose first constituent length $\ell$ satisfies
$q<\ell\le q+n$.
This extends a construction
in~\cite[Section 10]{IMS} for algebras of type $p$.

There will be one algebra for each value of a parameter $m$ in the range $0<m<n$.
Our conclusions are summarized in Theorem~\ref{thm:exceptional}.
This section constitutes a proof of that result,
but supplements it with the additional information that all constituents past the first of those algebras are ordinary and end with the same constant.
In this section we work under a less restrictive hypothesis on $n$
than the rest of the paper.

Before we start, we recall a known symmetry property of binomial coefficients modulo a prime.
A proof can be found in~\cite[Section~4]{Mat:binomial}
within a broader discussion,
and another one in~\cite[Lemma~4]{Mat:chain_lengths}.

\begin{lemma}\label{lucas_cor}
Let $p$ be any prime and let $q$ be a power of $p$.
Then for all non-negative integers $a$ and $b$ smaller than $q$
we have
\begin{equation*}
\binom{a}{q-1-b} \equiv  (-1)^{a+b} \binom{b}{q-1-a}\pmod{p}.
\end{equation*}
\end{lemma}

\subsection{A construction}
Let $K=F(t)$ be the field of rational functions on an indeterminate $t$ over $F$, and consider the ring of divided powers $K[x;c]$, where $q=p^c$.
Thus, a basis of $K[x;c]$ consists of the monomials $x^{(i)}$, with $0\le i<q$ (writing $1$ for $x^{(0)}$), with multiplication given by
$x^{(i)}x^{(j)}=\binom{i+j}{i}x^{(i+j)}$
and extended linearly.
Note that the binomial coefficient vanishes modulo $p$ when $i,j<q$ but $i+j\ge q$, and so the right-hand side may be read as zero in that case.
Also, it will be both safe and convenient for our calculations to allow divided power $x^{(j)}$ with negative $j$ and interpret them as zero.

Let $\partial$ be the standard derivation of $K[x;c]$, hence $\partial x^{(i)}=x^{(i-1)}$ for $0<i<q$ and $\partial 1=0$,
and let $I$ denote the identity map of $K[x;e]$.
View $K[x;c]$ as an abelian Lie algebra, and $\partial$ as an element of the general Lie algebra $\mathrm{gl}(K[x;c])$.
We first look at the Lie subalgebra of $\mathrm{gl}(K[x;c])$ generated by
$Z=-\partial-tx^{(q-1)}I$
and $tx^{(q-n)}I$.
Thus,
$Zx^{(i)}=-x^{(i-1)}$ for $0<i<q$, but $Z1=-tx^{(q-1)}$.
The reason for the negative signs here is to avoid some awkward alternating signs later, due to our use of left-normed long Lie products
conflicting with the prevalent notation where derivations act on the left of their arguments.

For $0<n\le q$ we have
$[tx^{(q-n)}I,Z]
=-tx^{(q-n)}I\partial+\partial(tx^{(q-n)}I)
=tx^{(q-n-1)}I$,
and then, by an easy induction,
\[
[tx^{(q-n)}I,Z^j]
=tx^{(q-n-j)}I
\]
for $j>0$.
Here $Z^j$ refers to our long Lie product convention, not to a compositional power of $Z$.
In particular, for $j=q-n$ we get $tI$, and hence
$tx^{(q-n)}I$ and $Z$ generate a Lie algebra of dimension $q-n+2$, metabelian and of maximal class.
Also, if we give the vector space (or abelian Lie algebra) $K[x;c]$ a cyclic grading over the integers modulo $q$ by assigning degree $j$ (modulo $q$) to $x^{(q-j)}$ ,
then $Z$ and $tx^{(q-n)}I$ are graded derivations of $K[x;c]$, of degrees $1$ and $n$, respectively.

We will now use this to construct a certain algebra
of type $n$, which depends on a certain parameter $m$ with $0<m<n\le q$.
Our construction will be more transparent if we initially look at the special case where $n=m+1$.
Thus, fix $0<m<q$,
and consider the semidirect sum of $K[x;c]$ and $\mathrm{gl}(K[x;c])$, hence with product
$[(f,A),(f',A')]=(Af'-A'f,[A,A'])$.
Inside that, consider the $F$-Lie subalgebra $\tilde L$ generated by
\[
z=(0,Z)
\qquad\text{and}\qquad
e_{m+1}=(x^{(q-1)},tx^{(q-m-1)}I).
\]
Then setting $e_j:=[e_{m+1},z^{j-m-1}]$ for $j>m+1$, we
inductively find
\[
e_j=(x^{(q+m-j)},tx^{(q-j)}I)
\]
for $m< j\le q+m$,
where the second entry is to be read as zero for $q<j\le q+m$ according to our stipulation on divided powers with negative exponents.
In particular, we have
$e_{q+m}=(1,0)$,
whence
$e_{q+m+1}=(tx^{(q-1)},0)$.
More generally, we find
\[
e_{rq+m+j}=(t^rx^{(q-j)},0)
\]
for $r>0$ and $0<j\le q$.

When we view the algebra $K[x;c]$ of divided powers
as an abelian Lie algebra over $F$ we can refine a shift
of its $\Z/q\Z$-grading over $K$ described earlier
to a $\Z$-grading over $F$,
by assigning degree $rq+m+j$ to $(t^rx^{(q-j)},0)$.
Then it can be easily checked that
$z$ and $e_{m+1}$ act on $K[x;c]$ as graded derivations
of degree $1$ and $m+1$ in this $\Z$-grading.
In particular, before doing any explicit calculation we see that
$[e_{m+1},e_j]$ is an $F$-scalar multiple of $e_{m+1+j}$ for all $j\ge q$.
Together with previous information on the $\Z/q\Z$-grading we see that this extends to any $j>m$.
Consequently, $\tilde L$ is a $\Z$-graded Lie algebra of maximal class, generated by elements of degree $1$ and $m+1$,
hence an algebra of type $m+1$ according to our terminology.

For any $n$ with $m<n\le q$, the subalgebra $L$ of $\tilde L$ generated by
$z=(0,Z)$
and
$e_n=(x^{(q+m-n)},tx^{(q-n)}I)$
is an algebra of type $n$.
We will now focus on that subalgebra $L$,
and determine its constituents
by finding the action of $e_n$ on it.
For $n<j\le q-n+m$ we have
\begin{align*}
[e_j,e_n]
&=
[(x^{(q+m-j)},tx^{(q-j)}I),(x^{(q+m-n)},tx^{(q-n)}I)]
\\&=
(tx^{(q-j)}x^{(q+m-n)}-tx^{(q-n)}x^{(q+m-j)},0)
=0.
\end{align*}
Next, for $0\le j<n$ we have
\begin{align*}
[e_{q+m-j},e_n]
&=
[(x^{(j)},tx^{(j-m)}I),(x^{(q+m-n)},tx^{(q-n)}I)]
\\&=
\left(\binom{q-n+j}{j-m}-\binom{q-n+j}{j}\right)
(tx^{(q-n+j)},0).
\end{align*}
We can write the coefficient of $(tx^{(q-n+j)},0)$, that is,
of $e_{q+m+n-j}$,
in a more convenient form by noting
\begin{align*}
\binom{q-n+j}{j-m}
=
\binom{q-n+j}{q-n+m}
&\equiv
(-1)^{j-m}\binom{n-1-m}{n-1-j}
\pmod{p}
\\&=
(-1)^{j-m}\binom{n-1-m}{j-m},
\end{align*}
where we have applied Lemma~\ref{lucas_cor} for the second step,
and similarly for the other binomial coefficient involved.
In conclusion, for $0\le j<n$ we have
\begin{equation}\label{eq:exceptional_first_const}
[e_{q+m-j},e_n]
=
(-1)^{j}\left((-1)^m\binom{n-1-m}{j-m}-\binom{n-1}{j}\right)
e_{q+m+n-j}.
\end{equation}

The above equation provides information on the first constituent
of the algebra of type $n$ under consideration, and we now determine
its length.
For $j=n-1$ the above equation reads
\[
[e_{q+m-n+1},e_n]=
(-1)^{n-1}\bigl((-1)^m-1\bigr)e_{q+m+1},
\]
and hence does not vanish if $m$ is odd.
However, if $m$ is even we have
\[
[e_{q+m-n+2},e_n]=
(-1)^{n+1}me_{q+m+2}.
\]
At this point we introduce the further assumption $m<p$
(although $m<2p$ would suffice for $p$ odd).
Under that assumption this Lie product does not vanish.
In conclusion, the highest $j$ in the range
$0\le j<n$ for which $[e_{q+m-j},e_n]$
does not vanish is $j=n-1$ when $m$ is odd,
and $j=n-2$ when $m$ is even.
Consequently, the first constituent of $L$ has length $q+m$ when $m$ is odd, and $q+m+1$ when $m$ is even.

\subsection{Further details on the structure of $L$}
Now we investigate $L$ further by looking at constituents past the first,
still under the assumptions $0<m<p$ and $m<n\le q$.
For $r>0$ we have
\begin{align*}
[e_{rq+m+j},e_n]
&=
[(t^rx^{(q-j)},0),(x^{(q+m-n)},tx^{(q-n)}I)]
\\&=
-\binom{2q-n-j}{q-n}
(t^{r+1}x^{(2q-n-j)},0),
\end{align*}
whence
\[
[e_{rq+m+j},e_n]=0
\qquad\text{for $0<j\le q-n$}
\]
and, after replacing $j$ with $q-j$,
\begin{equation}\label{eq:exceptional_later_const}
[e_{rq+q+m-j},e_n]=(-1)^{j+1}\binom{n-1}{j}e_{rq+q+m+n-j}
\qquad\text{for $0\le j<n$,}
\end{equation}
where we have used
\[
\binom{q-n+j}{q-n}
\equiv
(-1)^j\binom{n-1}{n-1-j}\pmod{p},
\]
due to to Lemma~\ref{lucas_cor}.
Note that the sequence $(\beta_i)_{i>n}$ of the algebra $L$
is periodic starting from $\beta_{q+1}$, with period length $q$.
We will return to this point after Proposition~\ref{prop:gen_func}
below.

We have just discovered that when $m$ is odd all constituents past the first
are ordinary of length $q$, ending in $-1$
(with our non-normalized choice of generators).
However, when $m$ is even all constituents past the first are ordinary,
ending in $-1$, but only those from the third on have length $q$,
with the second constituent having length $q-1$.
In reality, this discrepancy according to the parity of $m$
is less substantial than it may appear, and is due to the way we have
defined constituents: when $m$ is even the second constituent
has the same entries as an ordinary constituent, except for the
omission of an initial zero, which becomes
the last entry of the first constituent.

To further clarify the structure of the algebras
constructed above it will be helpful to introduce
generating functions.
More generally, the sequence $(\beta_i)_{i>n}$ of any algebra $L$ of type $n$ can be conveniently encoded
by its {\em generating function}
$\sum_{i>n}\beta_iX^i$,
a formal power series in $F[[X]]$,
where $X$ is an indeterminate.
Because
\[
[e_i,e_{n+1}]=[e_i,[e_n,z]]
=[e_i,[e_n,z]]-[e_i,z,e_n]
=(\beta_i-\beta_{i+1})e_{i+n+1},
\]
the coefficients of the series
$(1-1/X)\sum_{i>n}\beta_iX^i$
encode the adjoint action of $e_{n+1}$
on the basis elements $e_i$ for $i\ge n$.
Iterating this is a convenient way to find
the adjoint action of any basis element $e_j$, with $j\ge n$,
again on all basis elements $e_i$ for $i\ge n$:
that will be represented by the coefficients of
$(1-1/X)^{j-n}\sum_{i>n}\beta_iX^i$.
This device can be used to obtain the sequence
$(\tilde\beta_i)_{i>n+1}$ of the subalgebra $\tilde L$ of type $n+1$ of our algebra $L$ of type $n$,
from the sequence $(\beta_i)_{i>n}$:
\[
\sum_{i>n+1}\tilde\beta_iX^i=
(1-1/X)\biggl(\sum_{i>n}\beta_iX^i\biggr)+\beta_{n+1}X^n.
\]
Adding the term $\beta_{n+1}X^n$ is required to disregard
the action of $e_{n+1}$ on $e_n$.
Note that $\beta_{n+1}$ vanishes in many cases of interest,
and precisely as long as the
first constituent length $\ell$ of $L$ satisfies $\ell>2n$.

Now we produce the generating functions associated to the sequences
of the exceptional algebras $L$ constructed in this section,
as a more compact way of recording their multiplication.

\begin{prop}\label{prop:gen_func}
Assume $0<m<p$ and $m<n\le(q+m)/2$.
The generating function $\sum_{i>n}\beta_iX^i$
of the exceptional algebra $L$ of type $n$ constructed in this section, according to the parameter $m$, is
\[
\sum_{i>n}\beta_iX^i
=X^{q+m+1-n}(X-1)^{n-m-1}
-\frac{X^{q+m+1-n}(X-1)^{n-1}}{1-X^q}.
\]
\end{prop}

\begin{proof}
It is not difficult to read the generating function of $L$
straight off Equations~\eqref{eq:exceptional_first_const}
and~\eqref{eq:exceptional_later_const},
which describe the action of $e_n$.
Instead, we deal with the easier special case
where $n=m+1$, and then deduce the generating function
in the general case by considering subalgebras of higher types.

Thus, when $n=m+1$ Equation~\eqref{eq:exceptional_first_const}
yields
\[
\beta_{q+j}
=
(-1)^{{m-j}}\left((-1)^m\binom{0}{j}-\binom{m}{j}\right)
\qquad\text{for $0\le j\le m$,}
\]
and for $r>0$ Equation~\eqref{eq:exceptional_later_const}
yields
\[
\beta_{rq+q+j}=-(-1)^{m-j}\binom{m}{j}
\qquad\text{for $0\le j\le m$.}
\]
Together they show that the generating function of $L$ in this special case is
\[
\sum_{i>m+1}\beta_iX^i
=X^q
-\frac{X^{q}(X-1)^{m}}{1-X^q}.
\]

Multiplying this by $(1-1/X)^{n-m-1}$ we find
the generating function which encodes the action of $e_n$
in this algebra $L$, for any $n>m$:
\[
X^{q+m+1-n}(X-1)^{n-m-1}
-\frac{X^{q+m+1-n}(X-1)^{n-1}}{1-X^q}.
\]
According to a previous discussion,
this is the the generating function $(\tilde\beta_i)_{i>n}$
of the subalgebra $\tilde L$ of type $n$ of $L$,
as long as $n\le (q+m)/2$.
\end{proof}

As an application, we use the generating function
of Proposition~\ref{prop:gen_func} to find
the (unique) maximal abelian ideal
of the soluble algebras $L$ under consideration.
This could of course be deduced directly from our construction of $L$
in terms of the algebra of divided powers and its derivations,
but here we wish to emphasize how it can be quickly identified
from the generating function itself.
Since that maximal abelian ideal does not depend on $n$
(in the range $m<n\le(q+m)/2$ considered),
we may take $n=m+1$.
Thus, let $L$ be the exceptional algebra of type $m+1$
constructed in this section (which contains those of larger types
as subalgebras).
Then the adjoint action of a basis element $e_j$ on $L$, for any $j>m$,
is represented by the generating function
of Proposition~\ref{prop:gen_func} with $h$ in place of $n$.
In particular, the adjoint action of $e_{q+1}=(x^{(m-1)},0)$
on $L$ is represented by
$
X^{m}(X-1)^{q-m}
+X^{m}
$,
a polynomial of degree $q$.
This means that $[e_i,e_{q+1}]=0$ for $i>q$, and tells us that
the ideal of $L$ generated by $e_{q+1}$ is abelian, because
$[e_i,e_j]=0$ for all $i,j>q$
follows inductively via
$[e_i,e_{j+1}]=[e_i,e_j,z]-[e_i,z,e_j]$.
By contrast, the adjoint action of $e_{q}=(x^{(m)},tI)$
is represented by the generating function
$
X^{m+1}(X-1)^{q-m-1}
-X^{m+1}/(1-X)
$,
which is not a polynomial.
In fact, from that we can read off at once, in particular, that
$[e_i,e_q]=-e_{i+q}$ for $i>q$.

\section{Proof of Theorem~\ref{thm:final_contradiction}}\label{sec:final_contradiction}

Recall that, in general, the first constituent of
an algebra $L$ of type $n$
is the sequence
$\beta_n,\ldots,\beta_{\ell}$,
whose leading term  is $\beta_k$, where $k=\ell-n+1$,
and that the second constituent of $L$ is
$\beta_{\ell+1},\ldots,\beta_{\ell+\ell_2}$,
whose leading term is $\beta_{\ell+\ell_2-n+1}$.
Thus, the second constituent starts with precisely
$\ell_2-n$ zeroes.
However, the first constituent may, in general, end with some zeroes.
In this proof we will need to make the most use of this joint sequence of consecutive zeroes.

In order to prove Theorem~\ref{thm:final_contradiction}, assume now
$L$ is an algebra of type $n$ with
$k=q+k_0$, where $2\le k_0\le n-2$,
and $\ell_2 = q$ or $\ell_2 = q-1$.
We intend to obtain a contradiction from this.
Under such assumptions we have shown in
Subsection~\ref{subsec:spurious}
that the first constituent has at least
$k_0$ trailing zeroes.
(This came from the conclusion in Theorem~\ref{thm:polynomials}
that $x^{k_0}$ divides $g(x)$.)
However, the first constituent may have more than $k_0$ trailing zeroes.
As in Subsection~\ref{subsec:spurious}
let $\beta_{q+s}$ be the trailing term of the first constituent,
whence $k_0\le s<n$.

To keep notation more compact we introduce abbreviations
for a frequently occurring couple of indices and a related quantity.
Thus, let $\beta_M$ and $\beta_m$ be the trailing term
of the first constituent and the leading term of the second constituent,
respectively.
Hence $M=q+s$ and $m=q+k_0+\ell_2$, and we have
\begin{equation}\label{eq:betas}
\begin{aligned}
&&\beta_{i} &=  0  \text{ for $n\le i<k$},&
\beta_k &\ne 0,
\\
\beta_{M} &\ne 0,&
\beta_{i}  &=  0 \text{ for $M<i<m$},&
\beta_m  &\ne 0.
\end{aligned}
\end{equation}
So far we have no information on $\beta_i$
for $k<i<M$.
The strategy of our proof is to first fill this gap
by completely determining the first
constituent of $L$, and then to use this for a contradiction
against the second constituent.
A crucial quantity to this end is $m-M-1$, the
number of zeroes between $\beta_M$ and $\beta_m$, but it is more convenient to assign a name $t$ to the closely related quantity
\[
t  := m - M - n + 1
= k_0 + \ell_2 - s - n +1.
\]
Because the number of zeroes between $\beta_M$ and $\beta_m$
is at least $k_0+\ell_2-n$, and at most $\ell_2-1$, we have
\begin{equation}\label{eq:t_bounds}
\ell_2-2n+2+k_0\le t\le\ell_2-n+1.
\end{equation}
In particular, because our assumptions include $\ell_2 = q$ or $\ell_2 = q-1$, Equation~\eqref{eq:t_bounds} implies $q-2p<t<q$.
We will refine this crude information in Lemma~\ref{lemma_t}.

Before doing that we note the following generic application
of the Jacobi identity, which will occur repeatedly
throughout the proof, for $i>n$ and $h>0$:
\begin{align*}
[e_i,[e_{n+h},e_n]]
& = [[e_i, e_{n+h}], e_n] - [[e_i, e_n], e_{n+h}] \\
& = [[e_i, e_{n+h}], e_n] - \beta_i [e_{i+n}, e_{n+h}] \\
& = \sum_{g} (-1)^g \binom{h}{g} \beta_{i+g} [e_{i+h+n}, e_n] - \beta_i \sum_{g} (-1)^g \binom{h}{g} \beta_{i+n+g} e_{i+h+2n} \\
& = \left(\beta_{i+h+n} \sum_{g} (-1)^g \binom{h}{g} \beta_{i+g}  - \beta_i  \sum_{g} (-1)^g \binom{h}{g} \beta_{i+n+g} \right) e_{i+h+2n}.
\end{align*}
Hence whenever $\beta_{n+h}=0$, which means $[e_{n+h},e_n]=0$,
happens to hold for a certain $h$, its consequence
$[e_i,[e_{n+h},e_n]]=0$,
for any $i$,
is equivalent to the following quadratic equation in the
coefficients $\beta_j$, which we name for ease of reference:
\begin{equation}\label{eq:eih}
    E (i,h) : \qquad  \beta_{i+h+n} \sum_{g} (-1)^g \binom{h}{g} \beta_{i+g}  - \beta_i  \sum_{g} (-1)^g \binom{h}{g} \beta_{i+n+g} = 0.
\end{equation}
When applying Equation $E(i,h)$ we will choose $i$ and $h$
so that all coefficients $\beta_j$ involved vanish
aside from a few.

\begin{lemma}\label{lemma_t}
The integer $t$ is odd,
and $q - 2 p < t<  q - 1 $ with $t\neq q-p$.
Consequently, $(q-t)/2$ is a positive integer.
\end{lemma}

\begin{proof}
Noting $[e_{n+t-1}, e_n] = [e_{m-M}, e_n ] = 0$
because $m - M \leq k_0 - s + q \le q$,
we may use
$[e_M,[e_{m-M},e_n]]=0$,
which amounts to equation $E(M,t-1)$ in terms of the sequence of $L$.
Because all entries between $\beta_M$ and $\beta_m$ vanish,
equation $E(M,t-1)$ simply reads
\begin{align*}
0 & = \beta_{m} \beta_M - (-1)^{t-1} \beta_{M} \beta_{m} =
\bigl(1 + (-1)^{t}\bigr)\beta_{M} \beta_{m}.
\end{align*}
Since $\beta_M \beta_{m} \ne 0$ we find $1+(-1)^t\equiv 0\pmod{p}$,
and hence $t$ is odd because $p$ is odd.

From the fact that $s+n < 2 n < 2 p$ we infer
\begin{align*}
q - 2 p & < \ell_2 - s - n < t \leq \ell_2 - s \leq q - 2.
\end{align*}
Moreover $t \ne q -p$ because of the parity of $t$. It also follows that $(q-t)/2$ is a positive integer because $q-t$ is even and $1 < q-t$.
\end{proof}

\subsection{Computing $\beta_j$ for $q+2 \le j\leq q+s$}\label{subsec:k0_3}
Now we use
$[e_M,[e_{m-M+1},e_n]]=0$,
which means equation $E(M,t)$ in terms of the sequence of $L$.
Again because all entries between $\beta_M$ and $\beta_m$ vanish,
equation $E(M,t)$ reads
\begin{align*}
0 & = \beta_{M} \beta_{m+1} - \beta_{M} (t \beta_m - \beta_{m+1}) = \beta_M (2 \beta_{m+1} - t \beta_m).
\end{align*}
Since $\beta_M \ne 0$ we obtain
\begin{displaymath}
\beta_m =(2/t) \beta_{m+1}.
\end{displaymath}
In particular, $\beta_{m+1}$ is also nonzero.

In this subsection we will use a number of instances of the equation
$[e_{M-d},[e_{n+h},e_n]]=0$,
for a range of small nonnegative integers $d$,
to find linear relations between $\beta_M$ and
a bunch of preceding entries of the sequence.
This will then allow us to determine those entries in terms of $\beta_M$ alone.
However, because of alternating signs in the corresponding equations
$E(M-d,h)$, we will need to proceed in two ways
according to the parity of $d$, thus writing $d=2j$ or $d=2j+1$.
In each case we wish to choose $h$ so that
$\beta_{M-d+h+n}$, equivalently written $\beta_{m+1-t-d+h}$,
which is the last entry of the sequence
involved in the corresponding equations $E(M-d,h)$, is either $\beta_m$
or $\beta_{m+1}$.
For this reason we have just related those
to each other in the previous paragraph.
The following lemma determines the range of values of $j$
for which this will work.

\begin{lemma}\label{lemma_t_2}
We have
$[e_{n+t+2j}, e_n] =0$ for $0 \le 2j \le s-2$.
If $\ell_2 = q-1$
then this equation holds over the extended range
$0 \le 2j \le s-1$.
\end{lemma}

\begin{proof}
Because $t=k_0 - s + \ell_2-n+1$, we have
\begin{align*}
n+t+2j =  k_0 + \ell_2 + (2 j - s+1).
\end{align*}
Consequently, $n+t+2j \le k-1$ if $\ell_2 = q-1$
and $2j \le s-1$ or $\ell_2 = q$ and $2j \le s-2$.
Because $[e_i, e_n] = 0$ for $n \le i \leq k-1$,
we reach the desired conclusion.
\end{proof}

We start with the case of odd $d$, where the resulting linear equation
for the sequence of $L$ will be slightly simpler.
According to Lemma~\ref{lemma_t_2} we have
$[e_{n + t+2j},  e_n] = 0$ for $0 \le 2j \le s-2$.
(To avoid confusion we will consider only later the possibly
extended range in case $\ell_2=q-1$.)
Hence for $j$ in this range we can use
$[e_{M-2j-1},[e_{n+t+2j},e_n]]=0$.
Because all entries between $\beta_M$ and $\beta_m$ vanish,
the corresponding equation $E(i,h)$ with $i=M-2j-1$ and $h=t+2j$ reads
\begin{equation*}
\beta_{m} \sum_{g=0}^{2j+1} (-1)^g \binom{h}{g} \beta_{i+g}  - (-1)^{h} \beta_{i} \beta_m = 0.
\end{equation*}
Because $\beta_m \not = 0$ and $h$ is odd we conclude
\begin{equation}\label{ex_eq_1}
 \sum_{g=0}^{2j+1} (-1)^g \binom{t+2j}{g} \beta_{M-(2j+1)+g} = - \beta_{M-(2j+1)} \qquad \text{for $0 \le 2j \le s-2$}.
\end{equation}

Similarly, in case of even $d$ we can use
$[e_{M-2j},[e_{n+t+2j},e_n]]=0$,
again for $0\le 2j\le s - 2$ according to Lemma~\ref{lemma_t_2}.
Here, however, we assume $j>0$ because $j=0$
gives us an equation which we have already used
to find $\beta_m = (2/t)\beta_{m+1}$.
Thus, for $0<2j  \le s - 2$, the corresponding relation $E(i,h)$
with $i=M-2j$ and $h=t+2j$ reads
\begin{align*}
\beta_{m+1} \left( \sum_{g=0}^{2j} (-1)^g \binom{h}{g} \beta_{i+g} \right) - \beta_{i} \left(h \beta_m - \beta_{m+1} \right) = 0.
\end{align*}
Because $\beta_m = (2/t)\beta_{m+1}$ we have
\begin{align*}
\beta_{m+1} \left( \sum_{g=0}^{2j} (-1)^g \binom{h}{g} \beta_{i+g}  \right)
 - \beta_{m+1} \beta_{i} \left( 1 + \frac{4j}{t} \right)  = 0
\end{align*}
and because $\beta_{m+1} \ne 0$ we find
\begin{equation}\label{ex_eq_2}
\sum_{g=0}^{2j} (-1)^g \binom{t+2j}{g} \beta_{M-2j+g} = \left( 1 + \frac{4j}{t} \right) \beta_{M-2j}  \quad \quad \text{for $0<2j \le s-2$.}
\end{equation}

We will now show that Equations~\eqref{ex_eq_1} and~\eqref{ex_eq_2}
completely determine
$\beta_{q+1}, \dots, \beta_{M}$
if $s$ is even, and
$\beta_{q+2}, \dots, \beta_{M}$
if $s$ is odd,
in terms of $\beta_M$ alone.
This is because those equations form a
homogeneous linear system of rank one less than the
number of its unknowns.
We state this fact in a more general form,
so as to be able to apply it again later
to a different setting.

\begin{lemma}\label{lemma:linear_system}
The system of homogeneous linear equations
\begin{equation*}
\begin{cases}
\displaystyle
\sum_{g=0}^{2j+1} (-1)^g \binom{t+2j}{g} \gamma_{2j+1-g} = -\gamma_{2j+1}
&\text{for $0 \le 2j \le s-2$},
\\[1ex]
\displaystyle
\sum_{g=0}^{2j} (-1)^g \binom{t+2j}{g} \gamma_{2j-g} = \left( 1 + \frac{4j}{t} \right) \gamma_{2j}
&\text{for $0<2j \le s-2$},
\end{cases}
\end{equation*}
over the field of $p$ elements and in the indeterminates
$\gamma_0,\ldots,\gamma_{2\lfloor s/2\rfloor-1}$,
is equivalent to the linear system
\begin{equation*}
\gamma_i=(-1)^i \binom{(q-t)/2}{i}\gamma_0
\qquad\qquad\text{for $0\le i\le 2\lfloor s/2\rfloor-1$.}
\end{equation*}
\end{lemma}

\begin{proof}
The given linear system can be put into row-echelon form by
rearranging its equations, choosing alternately from either set,
so one sees that the equations are linearly independent.
They are in number of $2\lfloor s/2\rfloor-1$
(hence $s-1$ if $s$ is even, and $s-2$ if $s$ is odd).
Because that is precisely one less than the number
$2\lfloor s/2\rfloor$ of indeterminates,
the solutions form a one-dimensional space,
so it will be sufficient to check that setting
$\gamma_i=(-1)^i \binom{(q-t)/2}{i}$
for $0\le i\le 2\lfloor s/2\rfloor-1$
gives a particular solution.

Starting with the first set of equations,
for $0\le 2j\le s-2$ we find
\begin{align*}
\sum_{g=0}^{2j+1} (-1)^g \binom{t+2j}{g} \gamma_{2j+1-g}
&=
-\sum_{g=0}^{2j+1}\binom{t+2j}{g}
\binom{(q-t)/2}{2j+1-g}
\\=&
-\binom{(q+t)/2+2j}{2j+1}
=
-\binom{(q+t)/2+2j}{(q+t)/2-1}
\\\equiv&
\binom{(q-t)/2}{(q-t)/2-2j-1}
=
\binom{(q-t)/2}{2j+1}
=
- \gamma_{2j+1},
\end{align*}
where we have applied Vandermonde convolution
in the second line of the calculation,
and Lemma~\ref{lucas_cor}
in the third line
(where the congruence is modulo $p$).

The verification is similar
for the second set of equations, but with an
additional manipulation at the end.
Thus, for $0<2j\le s-2$ we find
\begin{align*}
\sum_{g=0}^{2j} (-1)^g \binom{t+2j}{g} \gamma_{2j-g}
&=
\sum_{g=0}^{2j} \binom{t+2j}{g}
\binom{(q-t)/2}{2j-g}
\\=&
\binom{(q+t)/2+2j}{2j}
=
\binom{(q+t)/2+2j}{(q+t)/2}
\\\equiv&
\binom{(q-t)/2-1}{(q-t)/2-2j-1}
=
\binom{(q-t)/2-1}{2j},
\end{align*}
where again we have applied Vandermonde convolution
in the second line of the calculation, and
Lemma~\ref{lucas_cor} in the third line.
Finally,
\[
\binom{(q-t)/2-1}{2j}
=
\frac{(q-t)/2-2j}{(q-t)/2}
\binom{(q-t)/2}{2j}
\equiv
\left( 1 + \frac{4j}{t} \right)
\gamma_{2j}
\]
completes our proof.
\end{proof}

Applying Lemma~\ref{lemma:linear_system}
to Equations~\eqref{ex_eq_1} and~\eqref{ex_eq_2},
with $\gamma_i=\beta_{M-i}$, gives us
$
\beta_{M-i}= (-1)^i \binom{(q-t)/2}{i} \beta_M
$
for $0 \le i \le 2\lfloor s/2\rfloor-1$.
Replacing $i$ with $s-i$,
these can be more conveniently written as
\begin{equation}\label{eq:system_sol_M}
\beta_{q+i}= (-1)^{s-i} \binom{(q-t)/2}{s-i} \beta_M
\end{equation}
for $1<i\le s$, and for $i=1$ as well if $s$ is even.
Thus, we have shown that
$\beta_{q+1}, \dots, \beta_{M}$
if $s$ is even, and
$\beta_{q+2}, \dots, \beta_{M}$
if $s$ is odd, are given by Equation~\eqref{eq:system_sol_M}
in terms of $\beta_M$ alone, which we know to be nonzero.
In the next subsection we will use this
to reach a contradiction,
under our hypothesis $2\le k_0\le n-2$.
The argument will be insufficient in a subcase of $k_0=2$,
for which we will need further considerations in following subsections.

\subsection{The general argument for a contradiction}
\label{subsec:k_large}

The coefficients $\beta_{q+i}$
determined by Equation~\eqref{eq:system_sol_M}
include $\beta_{q+k_0-1}$
except when $k_0=2$ and $s$ is odd.
However, $\beta_{q+k_0-1}=0$ by assumption,
and we will use this to reach a contradiction.
Thus, if $k_0 >2$, or if $k_0=2$ and $s$ is even,
then we have found
$\beta_{q+k_0-1} = \binom{(q-t)/2}{s-k_0+1} \beta_{q+s}$,
and this will give us the desired contradiction
provided that the binomial coefficient is not a multiple of $p$.

Because $(q-t)/2<p$ according to Lemma~\ref{lemma_t},
no $\binom{(q-t)/2}{s-i}$ is a multiple of $p$
for $0\le s-i\le(q-t)/2$.
Hence $\binom{(q-t)/2}{s-k_0+1}$ is not a multiple of $p$
as long as
$s-k_0+1\le(q-t)/2$, that is,
$2s \le q-t+2k_0-2$.
Because
$q-t+2k_0-2=q+k_0-\ell_2+n+s-3$
this condition reads
\begin{equation}\label{eq:sk03}
s \le q + k_0 - \ell_2 + n  -3.
\end{equation}
Because $s \le n-1$ and $\ell_2 \le q$,
Equation~\eqref{eq:sk03} is satisfied if $k_0 >2$,
or if $k_0=2$ and $s$ is even.
Thus, we obtain a contradiction in those cases.

If $k_0=2$, $s$ is odd, and $\ell_2 = q-1$,
then according to Lemma~\ref{lemma_t_2}
the ranges for $j$ in Equation~\eqref{ex_eq_1} and Equation~\eqref{ex_eq_2} extend to $0 \le 2j \le s-1$
and $0<2j \le s-1$, respectively.
Applying Lemma~\ref{lemma:linear_system}
with $s-1$ (which is even) in place $s-2$
gives us the values of $\gamma_i$ for $0\le i\le s$.
Consequently, in this case Equation~\eqref{eq:system_sol_M}
holds for $0\le i\le s$, and hence determines
$\beta_q$ and $\beta_{q+1}$ as well.
In particular, it implies that the latter is not zero
(because Equation~\eqref{eq:sk03} is satisfied),
and this again provides a contradiction.

In conclusion, we have found a contradiction
in all desired cases except for $k_0=2$
in case $s$ is odd and $\ell_2=q$.
To deal with that we need to extend our knowledge
of the coefficients $\beta_i$
to include the second constituent.

\subsection{Computing an initial portion of the second constituent}\label{subsec:second_const}

Recall that conditions~\eqref{eq:betas} hold.

Consider $j$ with $0 \le 2j \le s-2$.
According to Lemma~\ref{lemma_t_2} we have
$[e_{n+t+2j}, e_n] =0$.
Therefore, $E(i,h)$ with $i=M$ and $h=t+2j$ reads
\begin{align*}
   0 & = \beta_{m+2j+1} \sum_{g} (-1)^g \binom{h}{g} \beta_{M+g} - \beta_M \sum_{g} (-1)^g \binom{h}{g} \beta_{M+n+g} \\
   & = \beta_{m+2j+1} \beta_{M} - \beta_M \sum_{g} (-1)^{h-g} \binom{h}{h-g} \beta_{M+n+h-g}\\
   & = \beta_{m+2j+1} \beta_{M} + \beta_M \sum_{g} (-1)^{g} \binom{h}{g} \beta_{m+2j+1-g}.
\end{align*}
Because $\beta_M \ne 0$ we obtain
\begin{equation}\label{ex_eq_3}
    \sum_g (-1)^g \binom{t+2j}{g} \beta_{m+2j+1-g} = - \beta_{m+2j+1} \quad \text{for $0 \le 2j \le s-2$}.
\end{equation}

Now consider $j$ with $0 < 2j \le s-2$.
Then $E(i,h)$ with $i=M-1$ and $h=t+2j$ gives
\begin{align*}
   0 & = \beta_{m+2j} \sum_{g} (-1)^g \binom{h}{g} \beta_{M-1+g} - \beta_{M-1} \sum_{g} (-1)^g \binom{h}{g} \beta_{M-1+n+g} \\
   & = \beta_{m+2j} (\beta_{M-1} - h \beta_{M}) - \beta_{M-1} \sum_{g} (-1)^{h-g} \binom{h}{h-g} \beta_{M-1+n+h-g} \\
   & = \beta_{m+2j} \beta_{M-1} \left(1 + \frac{2h}{q-t} \right) + \beta_{M-1} \sum_{g} (-1)^g \binom{h}{g} \beta_{m+2j-g}.
\end{align*}
As in the previous case, because $\beta_{M-1} \ne 0$ we find
\begin{align*}
    0 & = \beta_{m+2j} \left( 1 + \frac{2t+4j}{-t} \right) + \sum_{g} (-1)^g \binom{t+2j}{g} \beta_{m+2j-g} \\
    & = \beta_{m+2j} \left( -1 - \frac{4j}{t} \right) + \sum_{g} (-1)^g \binom{t+2j}{g} \beta_{m+2j-g},
\end{align*}
namely
\begin{equation}\label{ex_eq_4}
    \sum_g (-1)^g \binom{t+2j}{g} \beta_{m+2j-g} = \beta_{m+2j} \left( 1 + \frac{4j}{t} \right) \quad \text{for $0 < 2j \le s-2$}.
\end{equation}

Applying Lemma~\ref{lemma:linear_system}
to Equations~\eqref{ex_eq_3} and~\eqref{ex_eq_4},
with $\gamma_i=\beta_{m+i}$, gives us
\begin{equation}\label{eq:system_sol_m}
\beta_{m+i}= (-1)^i \binom{(q-t)/2}{i} \beta_m \quad \text{for $0 \le i \le 2\lfloor s/2\rfloor-1$.}
\end{equation}

\subsection{Excluding $k_0=2$ with $s\neq n-1$.}\label{sec:2_n_1}
Recall from Subsection~\ref{subsec:k_large}
that we have reached a contradiction for $k_0>2$,
and also in some further cases.
Thus, from now on we assume $k_0=2$, $s$ odd and $\ell_2 = q$.
Note that this may only occur when $n$ is even,
because the first constituent length
$\ell=k+n-1=q+n+1$ is always even.
Our argument will leave out the extreme case where $s=n-1$,
which causes difficulties and will require separate
consideration in the final subsection.

Besides $M=q+s$ as before, here we have
$m=2q+2$ and
$t=q+3-s-n$.
Thus, $\beta_{q+s}$ is the trailing term of the first constituent, while $\beta_{2q+2}$ is the leading term of the second constituent.

To deal with this situation we expand vanishing
Lie brackets of the form $[e_r,e_r]$,
but we will need to distinguish two cases
according to the value of $s$.
If $2s \leq n$, then expanding the
right-hand side of the equation
$0 = [e_{q+s}, e_{q+s}]$ we get
\begin{align*}
0 & = \sum_{i} (-1)^i \binom{q+s-n}{i} \beta_{q+s+i} = \beta_{q+s},
\end{align*}
which contradicts our assumption $\beta_{q+s} \ne 0$.
This argument will not work for larger $n$
because the above sum will involve also $\beta_{2q+2}$ and
further, possibly nonzero entries of the second constituent.
Instead, if $2 s \ge n+2$ we expand the right-hand side of
$0 = [e_{q+s+1}, e_{q+s+1}]$ and find
\[
0 =\sum_{i} (-1)^i \binom{q+s+1-n}{i} \beta_{q+s+1+i}
=\pm \sum_i (-1)^i \binom{q+s+1-n}{q-s+1+i} \beta_{2q+2+i}.
\]
Here the sum does not involve the end of the first constituent,
but entries of the second, of which we have computed some
in the previous subsection in terms of $\beta_m$ alone.
Thus, using Equation~\eqref{eq:system_sol_m} we find
\[
0= \sum_i \binom{q+s+1-n}{2s-n-i}
\binom{(q-t)/2}{i} \beta_m.
\]
Note that here we are within the range of
validity $0\le i\le s-2$
of Equation~\eqref{eq:system_sol_m}
because the former binomial coefficient in the sum vanishes
unless $0\le i\le 2s-n$:
the latter range is contained the former
provided that $s\le n-2$, which is satisfied here because
we are assuming
$s\neq n-1$ in this subsection.
By Vandermonde convolution we find
\[
0= \binom{q+s+1-n+(q-t)/2}{2s-n} \beta_m,
\]
and because $\beta_m\neq 0$ a contradiction will follow once
we have checked that the binomial coefficient is not
a multiple of $p$.
In fact, that can be written as
\[
\binom{q+(3s-n-1)/2}{2s-n}
\equiv
\binom{(3s-n-1)/2}{2s-n}
\not\equiv 0
\pmod{p},
\]
where we have used Lucas' theorem after noting
$0<2s-n\le(3s-n-1)/2<p$.

\subsection{Excluding $k_0=2$ with $s = n-1$.}
In Subsection~\ref{sec:2_n_1} we dealt with the case $k_0=2$ and $s \ne n-1$.
In this section we will exclude the remaining case
where $k_0=2$ and $s =n-1$.
Recall that $n$ is even, and $n \ge 4$ because $2=k_0\le n-2$.
Moreover, $(q-t)/2 = n-2$.

Recall that $\beta_{2q+2}$ is the leading term of the second constituent, and that
\begin{equation}\label{eq:2q+2+i}
\beta_{2q+2+i}= (-1)^i \binom{n-2}{i} \beta_{2q+2} \quad \text{for $0 \le i \le n-3$}
\end{equation}
according to Equation~\eqref{eq:system_sol_m}.
We will now extend this information and prove the vanishing of a number of entries of the sequence past $\beta_{2q+n-1}$,
aiming at a final contradiction.

Note that the third constituent begins with $\beta_{2q+n+2}$.
According to Lemma~\ref{lemma:constituent_bound},
its length satisfies
$\ell_3 \ge\ell/2= (q+n+1)/2$,
which means $\beta_{2q+2+i} = 0$ for $n \le i < (q+n+1)/2$.
However, $2n$ falls within this range,
because $2n \le (q+n+1)/2-1$ is equivalent to $q-3n-1 \ge 0$,
which holds because
\begin{align*}
q-3n-1 & \geq p^2-3n-1  \ge (n+1)^2-3n-1 = n^2-n \ge 0.
\end{align*}
Consequently, we have found
\begin{equation}\label{2q2j2n}
\beta_{2q+2+i} = 0 \quad \text{for $n \le i \le 2n$.}
\end{equation}

We now show that the last two entries of the second constituent vanish, which means
$\beta_{2q+n} = \beta_{2q+n+1} = 0$.
In fact, because $\beta_{q+n}=0$ we have
\begin{align*}
0 =   [e_{q+n}, [e_n, z^{q}]] & = (\beta_{q+n} - \beta_{2q+n}) e_{2q+2n} = -\beta_{2q+n} e_{2q+2n},
\end{align*}
which implies $\beta_{2q+n} = 0$.
Similarly,
\begin{align*}
0  = [e_{q+n+1}, [e_n, z^{q+1}]] & = (-\beta_{2q+n+1} + \beta_{2q+n+2}) e_{2q+2n+2} = -\beta_{2q+n+1} e_{2q+2n+2}
\end{align*}
implies $\beta_{2q+n+1} = 0$ as well.

Now we prove by induction  that
\begin{equation}\label{2qn122i}
\beta_{2q+2n-1+2j} = \beta_{2q+2n+2j} = 0 \quad \quad \text{for $1 \le j \le \frac{q-2n+1}{2}$,}
\end{equation}
where we have already established the case $j=1$
within Equation~\eqref{2q2j2n}.
For $1 < j \le (q-2n+1)/2$ we have
$[e_{n+2j+1}, e_n]=0$. Hence, expanding Equation~\eqref{eq:eih} with $i=2q+n-1$ and $h=2j+1$, we get
\begin{align*}
\beta_{2q+2n+2j} \beta_{2q+n-1} - \beta_{2q+n-1} ((2j+1) \beta_{2q+2n-1+2j} - \beta_{2q+2n+2j}) = 0.
\end{align*}
Since $\beta_{2q+n-1}\neq 0$, we obtain
\begin{equation}\label{eq:2q56i_1}
2 \beta_{2q+2n+2j} - (2j+1) \beta_{2q+2n-1+2j}=0.
\end{equation}
Now we expand the equation $0 = [e_{q+h+n}, e_{q+h+n}]$ with $h = (n+2j)/2$. We get
\begin{equation}\label{eq:eqhn}
\begin{aligned}
    0 & = \sum_r (-1)^r \binom{q+h}{r} \beta_{q+h+n+r} \\
    & = \sum_{i}  (-1)^i \binom{h}{i} \left( \beta_{q+h+n+i} - \beta_{2q+h+n+i} \right)\\
    & = \pm (h \beta_{2q+2n+2j-1} - \beta_{2q+2n+2j}).
\end{aligned}
\end{equation}
In the course of this calculation we have split the first
summation range into two portions and used Lucas' theorem.
For the last step we have used
$\beta_{q+h+n+i}=0$  for $0 \le i \le h$,
due to
$q + n  \le  q+h+n+i
\le q+2h+n
=q+2n+2j
\le 2q+1$.
Thus, Equation~\eqref{eq:eqhn} gives us
$2\beta_{2q+2n+2j} = (n+2j) \beta_{2q+2n-1+2j}$.
Substituting this into Equation~\eqref{eq:2q56i_1} we find
$(n-1) \beta_{2q+2n-1+2j} = 0$,
whence $\beta_{2q+2n-1+2j} = 0$.
In turn, this yields $\beta_{2q+2n+2j} = 0$.
This completes the induction step, and proves
Equation~\eqref{2qn122i}.

So far we have proved
$\beta_{2q+n} = \dots = \beta_{3q+1}=0$.
We will need to compute two more entries of the sequence.
Because the above induction argument does not easily extend any further,
we resort to a slightly different calculation, expanding Equation $E(i,h)$ with $i=2q+n-2$ and $h=q-2n+4$:
\begin{align*}
0 & = \beta_{3q+2} (\beta_{2q+n-2} - (q-2n+4) \beta_{2q+n-1}) - \beta_{2q+n-2} (- \beta_{3q+2}) \\
& = 2  (\beta_{2q+n-2} + (n-2) \beta_{2q+n-1}) \beta_{3q+2}.
\end{align*}
Then, using Equation~\eqref{eq:2q+2+i} we get
\begin{align*}
0 & = 2  \left(\binom{n-2}{n-4} \beta_{2q+2}  - (n-2) \binom{n-2}{n-3} \beta_{2q+2} \right) \beta_{3q+2}\\
& = (n-2)(1-n) \beta_{2q+2} \beta_{3q+2}.
\end{align*}
Because $4 \le n < p$ we have $(n-2)(1-n)\beta_{2q+2} \ne 0$,
and $\beta_{3q+2} = 0$ follows.
Now note that the proof of Equation~\eqref{eq:2q56i_1} remains valid
when $j=(q-2n+3)/2$, in which case it reads
\[
2\beta_{3q+3}-(q-2n+4)\beta_{3q+2}=0,
\]
and $\beta_{3q+3} = 0$ follows as well.
In conclusion, we have shown
$\beta_{2q+n} = \dots = \beta_{3q+3} = 0$.

Now we are ready to produce a contradiction,
by expanding the equation $0 = \left[e_{q+h+n}, e_{q+h+n} \right]$ with
$h=(q+3-n)/2$.
Since this latter equation is analogous to Equation~\eqref{eq:eqhn} but with a different value for $h$,  we find
\begin{align*}
0 & = \sum_{i}  (-1)^i \binom{h}{i} \left( \beta_{q+h+n+i} - \beta_{2q+h+n+i} \right) \\
& = \pm \left( \binom{h}{h-1} \beta_{2q+2} - \beta_{2q+3} \right) \\
& = \pm \left(\frac{q+3-n}{2} \beta_{2q+2} + (n-2) \beta_{2q+2} \right) = \pm \frac{n-1}{2} \beta_{2q+2},
\end{align*}
where we have used Equation~\eqref{eq:2q+2+i} in the last line.
Regarding the summation in the second line
note that $\beta_{2q+h+n+i}=0$
for $0 \le i \le h$ due to
$2q + n  \le  2q+h+n+i \le 3q+3$.
In conclusion, the above calculation shows $\beta_{2q+2}=0$,
contradicting the fact that $\beta_{2q+2}$ is the leading term of the second constituent.
This contradiction completes the proof of
Theorem~\ref{thm:final_contradiction}.

\bibliography{References}

\end{document}